\documentclass[12pt]{amsart}

\usepackage{amsmath, amsfonts,amssymb,epsfig, hyperref}
\def\mathscr{\mathcal}

\usepackage{caption}

\usepackage[usenames]{color}

\theoremstyle{plain}
\newtheorem{theorem}{Theorem}[section]
\newtheorem{corollary}[theorem]{Corollary}

\newtheorem{lemma}[theorem]{Lemma}
\newtheorem{proposition}[theorem]{Proposition}

\newtheorem{quasi-theorem}[theorem]{Quasi-Theorem}
\newtheorem{conjecture}[theorem]{Conjecture}
\theoremstyle{definition}
\newtheorem{definition}[theorem]{Definition}

\newtheorem{question}[theorem]{Question}
\newtheorem{remark}[theorem]{Remark}

\newtheorem{hypothesis}[theorem]{Hypothesis}

\def\CC{\mathbb C}

\def\RR{\mathbb R}

\newcommand{\aff}{\mathbb{A}}

\newcommand{\rr}{\mathbb{R}}

\DeclareMathOperator{\Enc}{\mathbb{E}}
\DeclareMathOperator{\Sector}{\bf{S}}

\begin{document}

\title[Constructing buildings]{Constructing buildings and harmonic maps}
\author[Katzarkov \and Noll \and Pandit \and Simpson]{Ludmil Katzarkov \and Alexander Noll \and  Pranav Pandit \and Carlos Simpson}

%\institute{Universit\"at Wien, Fakult\"at f\"ur Mathematik, Oskar-Morgenstern-Platz 1, 1090 Wien, Austria
%\and University of Miami, Department of Mathematics, 1365 Memorial Drive, Ungar 515, Coral Gables, FL 33146 \email{l.katzarkov@math.miami.edu}
%\and CNRS UMR 7351, LJAD, Universit\'e Nice Sophia Antipolis, 06108 Nice, Cedex 2, France, \email{carlos@unice.fr}}

%\dedicatory{Happy Birthday Maxim!}

\begin{abstract}
In a continuation of our previous work \cite{KNPS},
we outline a theory which should lead to the construction of a universal pre-building and versal
building with a $\phi$-harmonic map from a Riemann surface, in the case of two-dimensional buildings
for the group $SL_3$. This will provide a generalization of the space of leaves of the foliation defined by
a quadratic differential in the classical theory for $SL_2$.
Our conjectural construction would determine the exponents for $SL_3$ WKB problems, and it can be
put into practice on examples.
\end{abstract}

\maketitle
%\tableofcontents

\section{Introduction}

Let $X$ be a Riemann surface, compact for now.
The moduli spaces of representations, vector bundles with connection, and
semistable Higgs bundles denoted $M_B$, $M_{DR}$ and $M_H$ respectively, are isomorphic as spaces, and we
denote the common underlying space just by $M$. This space
has three different algebraic structures, and the Betti and de Rham ones
share the same complex manifold.

These algebraic varieties are noncompact, indeed $M_B$ is affine, but they have compactifications.
First, $\overline{M}_{DR}$ and
$\overline{M}_H$ are natural orbifold compactifications, and when $M$ itself is smooth, the orbifold
compactifications are smooth.

The Betti moduli space, otherwise usually known as the character variety, admits many
natural compactifications.  Indeed the mapping class group acts on $M_B$ but it doesn't stabilize any one of them.
Gross, Hacking, Keel and Kontsevich \cite{GHKK} have recently studied more closely the problem of compactifying $M_B$ and show that
there are indeed some optimal choices with good properties.

Morgan-Shalen \cite{MorganShalen}
for $SL_2$, and recently Parreau \cite{Parreau}
for groups of higher rank, construct a compactification of $M$
where the points at $\infty$ are actions of the fundamental group on $\RR$-buildings. This compactification is related to the inverse limit of the algebraic compactifications of $M_B$, and we shall call it
$\overline{M}_B^P$.

Let $\overline{M}^{TSC}$ be the Tychonoff-Stone-\v{C}ech compactification.
It is universal, so it maps to the other ones.
The points at infinity may be identified with non-principal ultrafilters\footnote{Here, by an ultrafilter on a normal ($T_4$) topological space, we mean a maximal filter consisting of closed subsets, as were considered by Wallman \cite{Wallman}.}
$\omega$ on $M$.
We may therefore consider the maps
$$
\begin{array}{ccccc}
\overline{M}_{H} &\leftarrow &\overline{M}^{TSC}& \rightarrow& \overline{M}^P_B .\\
&  & \downarrow && \\
& & \overline{M}_{DR}  & &
\end{array}
$$
Given a non-principal ultrafilter $\omega$, consider its image $\omega _{DR}$
in $\overline{M}_{DR}$ (resp. $\omega _H$ in $\overline{M}_H$) and its image $\omega^P_B$ in $\overline{M}^P_B$.

Our basic question is to understand the relationship between $\omega_{DR}$ (resp. $\omega _H$) and
$\omega^P_B$.
In what follows we concentrate on $\omega _{DR}$ but there are also conjectures for $\omega _H$ as mentioned
in \cite{KNPS}, with recent results by Collier and Li \cite{CollierLi}.

The divisors at infinity in $\overline{M}_{DR}$ and $\overline{M}_H$ are both the same. They are
$$
D= M_H^{\ast} / \CC^{\ast}
$$
where $M_H^{\ast}:= M_H - {\rm Nilp}$ is the complement of the nilpotent cone. Recall that
${\rm Nilp}= f^{-1}(0)$ where $f: M_H \rightarrow \aff ^N$ is the Hitchin fibration.

Hence $\omega_{DR}$ (and similarly $\omega _H$) may be identified as an equivalence class of points in
$M_H^{\ast}$ modulo the action of $\CC^{\ast}$. We may therefore write
$$
\omega _{DR} = (E,\varphi )
$$
as a semistable Higgs bundle, such that the Higgs field $\varphi$ is not nilpotent, and this identification holds up to
nonzero complex scaling of $\varphi$.

It turns out that the essential features of the correspondence with $\overline{M}^P_B$ depend not on $\varphi$ up to
complex scaling, but up to real scaling. We therefore introduce the {\em real blowing up} $\widetilde{M}_{DR}$ of
$\overline{M}_{DR}$ along the divisor $D$. It is still a compactification of $M_{DR}$.
The boundary at infinity is $\widetilde{D}$ which is an $S^1$-bundle over $D$,
with
$$
\widetilde{D}= M_H^{\ast} / \RR^{+, \ast} .
$$
Let $\widetilde{\omega}_{DR}$ denote the image of $\omega$ in $\widetilde{D}\subset \widetilde{M}_{DR}$.
We may again write
$$
\widetilde{\omega} _{DR} = (E,\varphi )
$$
but this time $\varphi$ is defined up to positive real scaling.

We denote by $\phi$ the {\em spectrum} of $\varphi$. It consists of a multivalued tuple of holomorphic differential forms on
$X$. This is equivalent to saying that it is a point in the Hitchin base $\aff^N$, or again equivalently, a spectral curve $\Sigma \subset T^{\ast}X$. For us, it will be most useful to think of writing $\phi = (\phi _1,\ldots , \phi _r)$ locally over $X$, where $\phi _i$ are
holomorphic differential forms. There will generally be some singularities at points over which $\Sigma \rightarrow X$ is
ramified, and as we move around in the complement of these singularities the order of the $\phi _i$ may change.
Denote by $X^{\ast}$ the complement of the set of singularities.

Gaiotto-Moore-Neitzke \cite{GMN-SN} define the {\em spectral network} associated to $\phi$. This is one of the main players in our story.
We refer to \cite{GMN-SN, GMN-Snakes} for many pictures, and to our paper \cite{KNPS} for some specific pictures related to the BNR example.

These structures constitute our understanding of $\widetilde{\omega}_{DR}$.

On the other side, the point $\omega ^P_B$ may be identified with an action of $\pi _1(X)$ on an $\RR$-building
${\rm Cone}_{\omega}$. The theory of Gromov-Schoen and Korevaar-Schoen allows us to choose an equivariant
harmonic mapping $h: \widetilde{X}\rightarrow {\rm Cone}_{\omega}$. This is most often uniquely determined by
the action. The metric on the building, and hence the differential of the harmonic mapping, are well-defined
up to positive real scaling.

Here, we use the terminology ``harmonic map'' to an $\RR$-building to mean a map
such that the domain, minus a singular set of real codimension $2$,
admits an open covering where each open set maps into a single apartment, and these
local maps are harmonic mappings to Euclidean space. The differential of a harmonic map is the real part of a mutlivalued differential form.

In \cite{KNPS} we used the groupoid version of Parreau's theory \cite{Parreau}
to construct the
harmonic mapping $h$, and the classical local WKB approximation to show that its differential is $\phi$.

\begin{theorem}[\cite{KNPS}]
The differential of the harmonic mapping $h$ is the real part of the multivalued differential $\phi$ considered above:
$$
dh = \Re \phi .
$$
\end{theorem}

We note that Collier and Li have proven the corresponding result for the Hitchin WKB problem in some cases
\cite{CollierLi}.

This theorem suggests the following question.

\begin{question}
\label{mainq}
To what extent does $\phi$ determine $h$ and hence $\omega ^P_B$?
\end{question}

In what follows we shall assume that we have passed to the universal cover of the Riemann surface so $\widetilde{X}=X$.
Furthermore, since our considerations now, about how to integrate the differential $\phi$ into a harmonic map $h$,
concern mainly bounded regions of $X$, we may envision other examples.
Such might originally have come from noncompact Riemann surfaces and differential equations with irregular singularities.
The BNR example in \cite{KNPS} was of that form.

A {\em $\phi$-harmonic map} will mean a harmonic map $h$ from $X$ to a building $B$, such that its differential is $dh=\Re \phi$.

The goal of this paper is to sketch a theory which should lead to the answer to Question \ref{mainq},
at least for the group $SL_3$ and under certain genericity hypotheses about the absence of ``BPS states''.

\subsection{Contents}

We start by reviewing briefly the theory of harmonic maps to trees, viewed as
buildings for $SL_2$, with the universal map given by the leaf space of a foliation
defined by a quadratic differential. The remainder of the paper is devoted
to generalizing this classical theory
to higher rank buildings and in particular, as discussed in Section \ref{twodb},
to two-dimensional buildings for $SL_3$.

In order to go towards an algebraic viewpoint, we indicate in Section \ref{sec-presheaf}
how one can view a building as a presheaf on a certain Grothendieck site of {\em enclosures}.
This allows one to formulate a small-object argument completing a pre-building to a building.

In Section \ref{sec-initial} we describe the initial construction associated to a spectral
curve $\phi$. This is the building-like object, admitting a $\phi$-harmonic map from $X$,
obtained by glueing together small local pieces.

The initial construction will have points of positive curvature, and the main work is in
Sections \ref{sec-modifying} and \ref{sec-progress}
where we describe a sequence of modifications to the initial
construction designed to remove the positive curvature points. One of the main difficulties
is that a point where only four $60^{\circ}$ sectors meet, admits in principle two
different foldings. Information coming from the original harmonic map serves to determine
which of the two foldings should be used. In order to keep track of this information, we
introduce the notion of ``scaffolding''. The result of
Sections \ref{sec-modifying} and \ref{sec-progress} is a sequence of moves
leading to a sequence of two-manifold constructions. 

Our main conjecture is that when the spectral network \cite{GMN-SN} of $\phi$ doesn't have
any BPS states, this
sequence is well defined and converges locally to a two-manifold construction with only
nonpositive curvature. The universal pre-building is obtained by putting back some pieces that
were trimmed off during the procedure.

In Section \ref{BNRrev}, coming as an interlude between the two main sections,
we present an extended example. We show how to treat
the BNR example which we had already considered rather extensively in \cite{KNPS}, but from the
point of view of the general process being described here. It was
through consideration of this example that we arrived at our process, and we hope
that it will guide the reader to understanding how things work.

In Section \ref{sec-a2} we present just a few pictures showing what can happen in a typical slightly
more complicated example. This example and others will be considered more extensively elsewhere.

In Section \ref{sec-further} we consider some consequences of our still conjectural process, and
indicate various directions for further study.

\subsection{Convention} This paper is intended to sketch a picture rather than provide complete proofs. All stated theorems, propositions
and lemmas are actually ``quasi-theorems'': plausible and tractable statements for which we have in mind a potential method of proof.

Statements which should be considered as open
problems needing considerably more work to
prove, are labeled as conjectures.

\subsection{Acknowledgements}

The ideas presented here are part of our ongoing attempt to understand the geometric
picture relating points in the Hitchin base and stability conditions. So this work is initiated
and motivated by the work of Maxim Kontsevich and Yan Soibelman, as exemplified by their many lectures
and the discussions we have had with them. It is a great pleasure to dedicate this paper to Maxim on the
occasion of his $50^{\rm th}$ birthday.

The reader will note that many other mathematicians are also contributing to this fast-developing theory
and we would also like to thank them for all of their input. We hope to present here a small contribution of
our own.

On a somewhat more specific level, the idea of looking at harmonic maps to buildings, generalizing
the leaf space tree of a foliation, came up during some fairly wide discussions about this developing theory in which Fabian Haiden also participated and made important comments, so we would like to thank Fabian.

We would like to thank many other people including Brian Collier, Georgios Daskalopoulos, Mikhail Kapranov, Fran\c{c}ois Labourie,  Ian Le, Chikako Mese, Richard Schoen, Richard Wentworth, and the members of the
Geometric Langlands seminar at the University of Chicago, for interesting
discussions contributing to this paper. We thank Doron Puder for an interesting talk at the IAS on
the theory of Stallings graphs.

The authors would like to thank the University of Miami for hospitality and support during the completion of this work. The fourth named author would in addition like to thank the Fund For Mathematics at the Institute for Advanced Study in Princeton for support.  The first named author was supported by the Simons Foundation as a Simons Fellow.

The first, second, and third named authors were funded by an ERC grant, as well as the following grants: NSF DMS 0854977 FRG, NSF DMS 0600800, NSF DMS 0652633 FRG, NSF DMS 0854977, NSF DMS 0901330, FWF P 24572 N25, FWF P20778, FWF P 27784-N25. The fourth named author was supported in part by the ANR grant 933R03/13ANR002SRAR (Tofigrou). The first named author was also funded by the following grants: DMS-1265230 ``Wall Crossings in Geometry and Physics'', DMS-1201475 ``Spectra, Gaps, Degenerations and Cycles'', and OISE-1242272 PASI ``On Wall Crossings, Stability Hodge Structures \& TQFT''. The second named author was in addition funded by the Advanced Grant ``Arithmetic and physics of Higgs moduli spaces'' No. 320593 of the European Research Council.

We would all like to thank the IHES for hosting Maxim's birthday conference, where
we did a significant part of the work presented here.

\section{Trees and the leaf space}

Let us first consider the case of representations into $SL_2$. Exact WKB analysis and
its relationship with the geometry of quadratic differentials has been considered extensively
by Iwaki and Nakanishi \cite{IwakiNakanishi}.

In the Morgan-Shalen-Parreau theory, the limiting building ${\rm Cone}_{\omega}$ is then an $\RR$-tree.
In saying {\em $\RR$-tree} we include the data of a distance function which is the standard one on any apartment.
The choice of metric is well-defined up to
a global scalar on ${\rm Cone}_{\omega}$.

The spectral curve $\Sigma \rightarrow X$ is a $2$-sheeted covering defined by a quadratic differential $q\in H^0(X,K_X^{\otimes 2})$.
The multivalued differential form $\phi$ is just the set of two square roots: $\phi _i = \pm \sqrt{q}$. The singularities are
the zeros of $q$, and we shall usually assume that these are simple zeros. It corresponds to saying that $\Sigma$ is smooth
which, for a $2$-sheeted cover, implies having simple branch points.

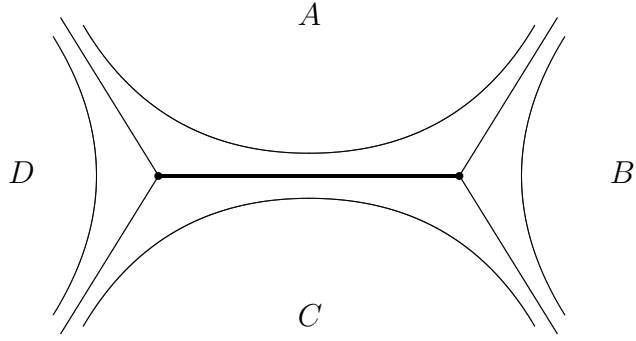
\begin{figure}
$$
\setlength{\unitlength}{.5mm}
\begin{picture}(200,120)

\put(60,60){\circle*{2}}
\put(140,60){\circle*{2}}

\linethickness{.5mm}
\qbezier(60,60)(100,60)(140,60)
\thinlines

\qbezier(60,60)(47,81)(34,102)
\qbezier(60,60)(47,39)(34,18)

\qbezier(32,97)(55,60)(32,23)

\qbezier(140,60)(153,81)(166,102)
\qbezier(140,60)(153,39)(166,18)

\qbezier(168,97)(145,60)(168,23)

\qbezier(40,100)(60,66)(100,66)
\qbezier(160,100)(140,66)(100,66)

\qbezier(40,20)(60,54)(100,54)
\qbezier(160,20)(140,54)(100,54)

\put(97,100){$A$}
\put(180,58){$B$}
\put(97,20){$C$}
\put(20,58){$D$}

\end{picture}
$$
\caption{A foliation with BPS state}
\label{foliation1}
\end{figure}

The differential form $\Re \phi$ is well defined up to a change of sign. Therefore, it defines a single real direction at every nonsingular
point of $X$,
which we call the {\em foliation direction}. These directions are the tangent directions to the leaves of a real foliation, which we
call the {\em foliation defined by $\phi$} or equivalently the {\em foliation defined by the quadratic differential $q$}.

Suppose $T$ is an $\RR$-tree and $h: \widetilde{X}\rightarrow T$ is a $\phi$-harmonic map. Then, the closed leaves of
the foliation defined by $\phi$ map to single points in $T$. This is clear from the differential condition at smooth points.
By continuity it also extends across the singularities: the closed leaves are defined to be the smallest closed subsets which are
invariant by the foliation. In pictures, a leaf entering a three-fold singular point therefore generates two branches of the
leaf going out in the other two directions, and all three of these branches have to map to the same point in $T$.

We assume that the {\em space of leaves of the foliation} is well-defined as an $\RR$-tree.
Denote it by $T _{\phi}$. The points of $T_{\phi}$ are by definition just the closed leaves of the foliation.

In order to consider the universal property, we say that a {\em folding map} $u:\RR \rightarrow T '$ is a map such that
there is a locally finite decomposition of $\RR$ into segments, such that on each segment $u$ is an isometric
embedding. A map $u:T \rightarrow T'$ between $\RR$-trees is a {\em folding map} if its restriction to each apartment
is a folding map. Note that an isometric embedding is a particularly nice kind of folding map which doesn't actually fold anything.

\begin{corollary}
The projection map $h_{\phi}: X\rightarrow T _{\phi}$ is a $\phi$-harmonic map, universal among $\phi$-harmonic
maps to $\RR$-trees. That says that if $h:X\rightarrow T$ is any other $\phi$-harmonic map, there exists a unique
folding map $T _{\phi}\rightarrow T$ making the diagram
$$
\begin{array}{ccc}
X & \rightarrow & T _{\phi} \\
& \searrow & \downarrow \\
& & T
\end{array}
$$
commute.
\end{corollary}

In general, we have to admit the possibility that $T_{\phi} \rightarrow T$ be a folding map rather than an isometric
embedding. Let us look at an example which illustrates this. Suppose that our quadratic differential or equivalently $\phi$, has
two singular points which are on the same leaf of the foliation. See Figure \ref{foliation1}.

The leaf segment between the two singularities is emphasized, it is
called a {\em BPS-state} \cite{GMN-SN}. These play a key
role in the wallcrossing story. Here, it is the existence of the BPS state which leads to nonrigidity of the harmonic map.

The space of leaves of the foliation is a tree $T_{\phi}$ with four segments. In Figure \ref{tree1}
these segments are labeled in the same way as the
corresponding regions in the previous picture of the foliation.

\begin{figure}
$$
\setlength{\unitlength}{.5mm}
\begin{picture}(200,120)

\put(100,60){\circle*{2}}

\put(100,60){\line(0,1){30}}
\put(100,60){\line(1,0){70}}
\put(100,60){\line(0,-1){30}}
\put(100,60){\line(-1,0){70}}

\put(97,105){$A$}
\put(180,58){$B$}
\put(97,15){$C$}
\put(20,58){$D$}

\end{picture}
$$
\caption{The quotient tree}
\label{tree1}
\end{figure}
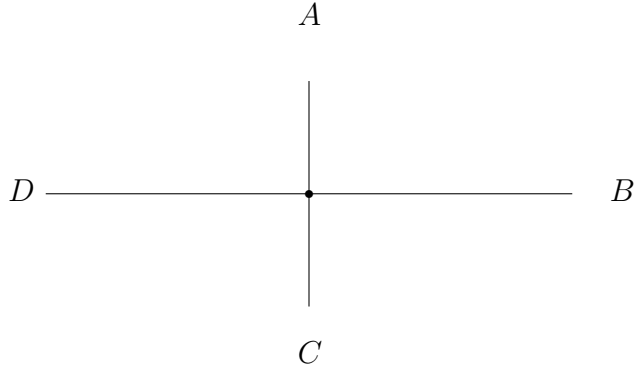

Suppose we have a $\phi$-harmonic map $h$ to another tree $T$. The universal property
gives a factorizing map $u:T_{\phi}\rightarrow T$.

The map $h$ is supposed to send small open sets in $X^{\ast}$ to single
apartments in $T$. The small open sets in $X$ provide some small segments in the tree $T_{\phi}$, and the
segments in $T_{\phi}$ which are images of these neighborhoods, have to map by isometric embeddings
(i.e. not be folded) under $u$.

Little neighborhoods
along the segments joining the regions for example $A$ to $B$, $B$ to $C$, $C$ to $D$ and $D$ to $A$,
as well as a neighborhood along the BPS state,
are shown in Figure \ref{littlenbds}.

\begin{figure}
$$
\setlength{\unitlength}{.5mm}
\begin{picture}(200,120)

\put(60,60){\circle*{2}}
\put(140,60){\circle*{2}}

\linethickness{.5mm}
\qbezier(60,60)(100,60)(140,60)
\thinlines

\qbezier(60,60)(47,81)(34,102)
\qbezier(60,60)(47,39)(34,18)

\qbezier(32,97)(55,60)(32,23)

\qbezier(140,60)(153,81)(166,102)
\qbezier(140,60)(153,39)(166,18)

\qbezier(168,97)(145,60)(168,23)

\qbezier(40,100)(60,66)(100,66)
\qbezier(160,100)(140,66)(100,66)

\qbezier(40,20)(60,54)(100,54)
\qbezier(160,20)(140,54)(100,54)

\put(97,100){$A$}
\put(180,58){$B$}
\put(97,20){$C$}
\put(20,58){$D$}

\put(153,81){\circle{18}}
\put(153,39){\circle{18}}
\put(47,81){\circle{18}}
\put(47,39){\circle{18}}

\put(100,60){\circle{18}}

\end{picture}
$$
\caption{Some little neighborhoods}
\label{littlenbds}
\end{figure}
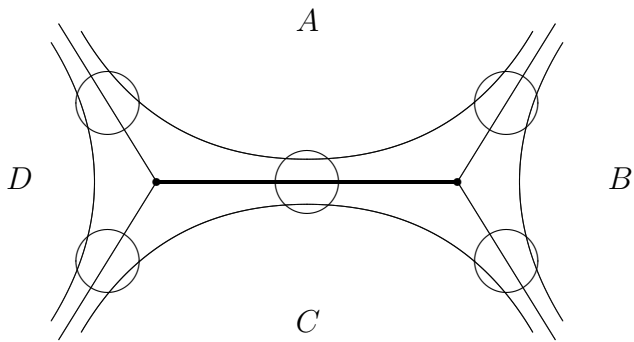

The four outer neighborhoods project to four corner segments as shown in Figure \ref{cornersegs}.
\begin{figure}
$$
\setlength{\unitlength}{.5mm}
\begin{picture}(200,120)

\put(100,60){\circle*{2}}

\put(100,60){\line(0,1){30}}
\put(100,60){\line(1,0){70}}
\put(100,60){\line(0,-1){30}}
\put(100,60){\line(-1,0){70}}

\put(97,105){$A$}
\put(180,58){$B$}
\put(97,15){$C$}
\put(20,58){$D$}

\qbezier(102,75)(102,62)(115,62)
\qbezier(102,45)(102,58)(115,58)
\qbezier(98,75)(98,62)(85,62)
\qbezier(98,45)(98,58)(85,58)

\end{picture}
$$
\caption{Non-folded segments}
\label{cornersegs}
\end{figure}
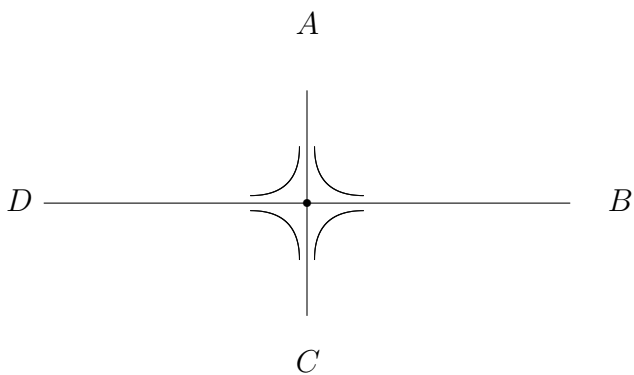
Therefore $u$ cannot fold these segments.

The little neighborhood along the central leaf projects to the segment shown in Figure \ref{cseg} joining the upper edge $A$ of the tree and the lower edge $C$.
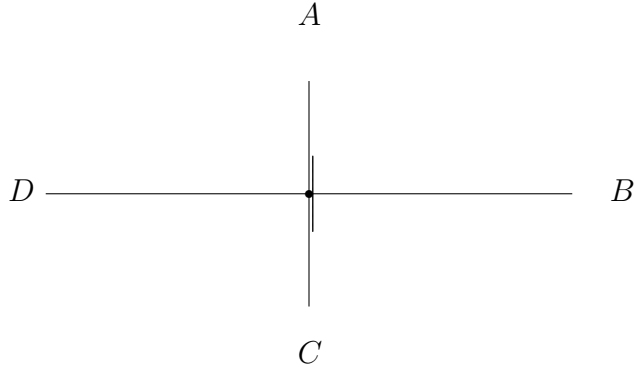
\begin{figure}
$$
\setlength{\unitlength}{.5mm}
\begin{picture}(200,120)

\put(100,60){\circle*{2}}

\put(100,60){\line(0,1){30}}
\put(100,60){\line(1,0){70}}
\put(100,60){\line(0,-1){30}}
\put(100,60){\line(-1,0){70}}

\put(97,105){$A$}
\put(180,58){$B$}
\put(97,15){$C$}
\put(20,58){$D$}

\qbezier(101,70)(101,60)(101,50)

\end{picture}
$$
\caption{Non-folded segment from central neighborhood}
\label{cseg}
\end{figure}
It follows that $u$ is not allowed to fold the two segments $A$ and $C$ together.

However, $u$ could very well fold the segments $B$ and $D$ together since they are not constrained by
any neighborhoods in $X$. Therefore, we can have a
harmonic\footnote{R. Wentworth pointed out to us that our composed map
to the tree in Figure \ref{newtree}, while locally harmonic, will not however
be globally energy-minimizing.}
map to the tree $T$ shown in Figure \ref{newtree}.
\begin{figure}
$$
\setlength{\unitlength}{.5mm}
\begin{picture}(200,120)

\put(100,60){\circle*{2}}

\put(100,60){\line(0,1){30}}
\put(85,45){\line(1,0){55}}
\put(100,60){\line(0,-1){14}}
\put(100,30){\line(0,1){14}}
\put(85,45){\line(-1,0){55}}

\put(100,60){\line(-1,-1){15}}

\put(97,105){$A'$}
\put(151,43){$B'$}
\put(97,15){$C'$}
\put(18,43){$D'$}

\end{picture}
$$
\caption{A folded tree}
\label{newtree}
\end{figure}
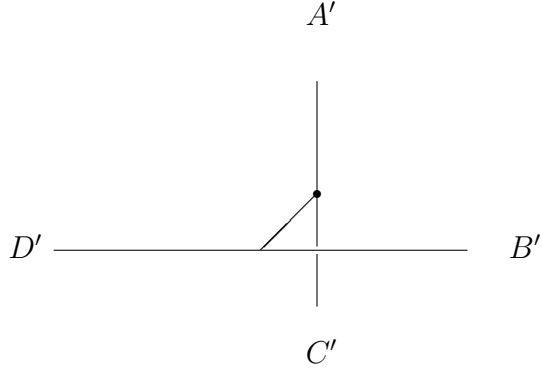
The central dot goes to the central dot, the segments $A$ and $C$ go to $A'$ and $C'$, and the
segments $B$ and $D$ go to the paths from the central dot out to $B'$ and $D'$, which are folded together along
some short segment.

We see in this example that the projection to the leaf space of the foliation is not rigid, and this non-rigidity looks closely
related to the presence of the BPS state.

On the contrary,
if there are no BPS states, then $h_{\phi}$ is rigid:

\begin{theorem}
\label{treecase}
Suppose that the spectral curve is smooth, i.e. the quadratic differential has simple zeros.
If there are no BPS states, then any $\phi$-harmonic map $h:X\rightarrow T$ factors as $h=u\circ h_{\phi}$
through a unique map
$u:T_{\phi}\rightarrow T$ which is an isometric embedding. Thus, $T$ is just $T_{\phi}$ plus some other edges
not touched by the image of $h$.
\end{theorem}

In the situation of the theorem, any two nonsingular points of $X$ are joined by a strictly noncritical path, that is to say
a path transverse to the foliation. The distance between the two points in $T_{\phi}$ is the length of the path using the
transverse measure defined by $\Re \phi$. This noncritical path also has to go to a noncritical path in $T$, so the
distance in $T$ is the same as in $T_{\phi}$.

If there is a BPS state, then some distances might not be well-determined. For example
the distance between points in regions $B$ an $D$ may change with the family
of maps pictured in Figure \ref{newtree}.

Non-uniqueness of the distance is something that has to be expected from the point of view of the
Voros resurgent expression for the transport function \cite{Voros}. The resurgent expression
may be viewed, roughly speaking, as a combination of contributions of different
exponential orders $e^{\lambda _it}$. See \cite{asymptotics} for a more precise
discussion of how this works.
When there is a BPS state,
then two exponents have the same real part, $\Re \lambda _i =\Re \lambda _j$. If
these are the leading terms, then we obtain a function which becomes oscillatory along
the positive real direction in $t$. A simple example would be $2\cos (t) = e^{it} + e^{-it}$.
If the transport function $T_{PQ}(t)$ is oscillatory, then the asymptotic exponent
$$
\nu_{PQ}^{\omega} = \lim _{\omega} \frac{1}{t} \log \| T_{PQ}(t) \|
$$
could very well depend on the choice of ultrafilter $\omega$ used to define the limit.
Thus, when there is a BPS state we can expect that the distance function defined
by the harmonic map could depend on
the choice of ultrafilter. In particular, it wouldn't be uniquely determined by the
spectral differential $\phi$.

These considerations, fitting perfectly well with the
illustration of Figure \ref{newtree}, motivate the hypothesis about BPS states in
Theorem \ref{treecase}.

\section{Two-dimensional buildings}
\label{twodb}

We feel that there should be a similar picture
for higher-dimensional buildings. The basic philosophy and motivations were described in
\cite{KNPS}.

Our idea at the current stage of this project is to concentrate on mappings to $2$-dimensional buildings.
These buildings are asymptotic cones for symmetric spaces $SL_3/SU(3)$, and the
mappings are limiting points in the sense of Parreau's theory \cite{Parreau} \cite{KNPS}.
The spectral curves for such mappings are triple coverings $\Sigma \rightarrow X$.

For this $SL_3$ case, our goal is to sketch the outlines of a theory which should lead to
a generalization of Theorem \ref{treecase}. It will say that
for generic $\phi$ in a chamber where there are no BPS states,
there should be a versal map to a building, such that the resulting distance
function is uniquely determined by $\phi$ and preserved under $\phi$-harmonic maps.

The $SL_3$
situation is in many ways an intermediate case. In the $SL_2$ case, the mapping $h_{\phi}$ to the
tree $T _{\phi}$ was {\em surjective}. For $SL_r$ with $r\geq 4$, the corresponding buildings have dimension $\geq 3$, so
the map from $X$ has no hope of being surjective. We will present some speculations about that higher rank situation at the end of the paper.

In the $SL_3$ case, the dimension of the building is two, which is the same as the dimension of $X$. Therefore, we can expect that
$X$ will surject onto a subset which at least has nonempty interior. So, it presents some similarity with the case of trees, and
this simplifies the geometrical aspects. We are able to develop a fairly precise although still conjectural picture.

The image of $X$ will be a quotient of $X$, glueing together
points over certain portions. This was seen in our BNR example of \cite{KNPS} which shall be recalled
in detail in Section \ref{BNRrev} below. Our goal is to
construct a map
$$
h_{\phi}: X\rightarrow B _{\phi}
$$
which should play the same role as the projection to the tree of leaves in the $SL_2$ case.

The construction has several steps. The main part will be the construction of a map to a {\em pre-building}
$$
h_{\phi}^{\rm pre}: X\rightarrow B _{\phi}^{\rm pre}
$$
such that the building $B _{\phi}$ can then be obtained from $B _{\phi}^{\rm pre}$ by adding on sectors not touched
by the image of $X$. Already in the BNR example of \cite{KNPS} there were infinitely many additional sectors to be added here.
They seem to be somewhat less related to the geometry of the situation.

The pre-building will itself be a quotient of an {\em initial construction}. The initial construction is obtained by glueing together
small pieces. Recall that one of the main characteristics of a building is its nonpositive curvature property.
The initial construction will, however, have some positively curved points: those are points where the total surrounding
angle is $240^{\circ}$ rather than $360^{\circ}$. As we shall see below, it leads to a process of successive pasting together
of parts of the construction. We conjecture that after a locally finite number of steps this process should stop and  give a
pre-building.

\section{Constructions as presheaves on enclosures}
\label{sec-presheaf}

In order to get started,
we need a precise way to manipulate the building-like objects involved in the construction. The idea for passing from
$B _{\phi}^{\rm pre}$ to $B _{\phi}$ will be to apply the small object argument. Also, the construction of
$B _{\phi}^{\rm pre}$ itself will involve successively imposing a bigger and bigger relation on the initial construction. So, it appears that we are working with algebraic rather than topological or metric objects.  This makes it desirable to have an algebraic framework.

We propose
to consider a {\em Grothendieck site} $\Enc$ of the basic building blocks, called ``enclosures''. Then ``constructions'' will be sheaves
of sets on the site $\Enc$, satisfying basic local presentability and separability properties.

Intuitively, a construction is a space obtained by glueing together the basic pieces such as shown in
Figure \ref{someencs}.

Proofs are not yet given, however we hope that they will be reasonably straightforward.
The general theory is described for buildings of any dimension.

Let $A$ be the affine space on which our buildings will be modeled.
A {\em root half-space} is a half-space bounded by a root hyperplane. 
An {\em enclosure} is a bounded closed subset defined by the intersection of finitely many root half-spaces.
For buildings corresponding to $SL_3$, the affine space 
is\footnote{More precisely $A$ is the space of triples $(x_1,x_2,x_3)$ with $\sum x_i=0$ and the root
half-spaces are defined by $x_i-x_j\geq c$.}  
$A=\RR^2$ and
some examples of enclosures are shown in Figure \ref{someencs}.
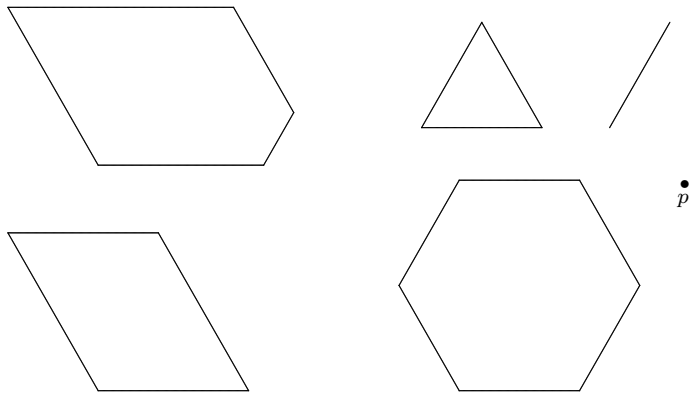
\begin{figure}
$$
\setlength{\unitlength}{.5mm}
\begin{picture}(200,120)

\qbezier(0,112)(12,91)(24,70)
\qbezier(0,112)(30,112)(60,112)
\qbezier(24,70)(46,70)(68,70)
\qbezier(60,112)(68,98)(76,84)
\qbezier(68,70)(72,77)(76,84)

\qbezier(110,80)(126,80)(142,80)
\qbezier(110,80)(118,94)(126,108)
\qbezier(142,80)(134,94)(126,108)

\qbezier(160,80)(168,94)(176,108)

\put(180,65){\circle*{2}}
\put(178,60){${\scriptstyle p}$}

\qbezier(0,52)(12,31)(24,10)
\qbezier(0,52)(20,52)(40,52)
\qbezier(40,52)(52,31)(64,10)
\qbezier(24,10)(44,10)(64,10)

\qbezier(120,10)(136,10)(152,10)
\qbezier(152,10)(160,24)(168,38)
\qbezier(120,10)(112,24)(104,38)

\qbezier(120,66)(136,66)(152,66)
\qbezier(152,66)(160,52)(168,38)
\qbezier(120,66)(112,52)(104,38)

\end{picture}
$$
\caption{Some enclosures}
\label{someencs}
\end{figure}

An {\em affine map} of enclosures $E\rightarrow E'$ is a map which is the restriction of an
automorphism of $A$ given by
an affine Weyl group element. Let $\Enc$ be the category of enclosures and affine maps between them.
There is an object ``point'' denoted $p$ consisting of a single point in $A$.

We define a Grothendieck topology on $\Enc$ as follows: a covering of $E$ is a finite collection of affine maps of enclosures
$E_i\rightarrow E$ such that $E$ is the union of their images.

The category $\Enc$ admits fiber products, but not products and in particular there is no terminal object.

\begin{proposition}
The coverings define a Grothendieck topology on $\Enc$.
\end{proposition}

If $E$ is an enclosure, we denote by $\widetilde{E}$ the sheaf associated to the presheaf represented by
$E$. It is different from the presheaf: if $E'$ is another enclosure, then the sections of $\widetilde{E}$ over $E'$,
which is to say the maps $E'\rightarrow \widetilde{E}$ or equivalently the maps $\widetilde{E'}\rightarrow \widetilde{E}$, are
the {\em folding maps} from $E'$ to $E$. These are the continuous maps which are piecewise affine for a decomposition of $E'$
into finitely many pieces which are themselves enclosures.

We can give a normalized form for coverings.
Suppose $H_1,\ldots , H_k$ is a sequence of parallel Weyl hyperplanes in order. Then we obtain a sequence of {\em strips}
$S_0,\ldots , S_{2k}$ covering $A$. The strip $S_{2i}$ is the closed subset consisting of
everything between and including $H_i$ and $H_{i+1}$, with $S_0$ and $S_{2k}$ being
the two outer half-planes (we assume $k\geq 2$ so there is no question about the ordering of these).
The strip $S_{2i+1}$ is just $H_{i-1}$ itself.
Suppose we are given a collection of such sequences of strips $S_{\cdot}^1,\ldots , S_{\cdot}^a$ for various directions of
the Weyl hyperplanes. Then for $J=(j_1,\ldots , j_a)$ with $0\leq j_i\leq 2k_i$ we may consider the enclosure
$$
U_J:=E\cap  S^1_{j_1}\cap \cdots \cap S^a_{j_a}.
$$
These cover  $E$.

\begin{lemma}
Suppose $\{ V_k\}$ is a covering of $E$. Then it may be refined to a
{\em standard covering}, that is to say a covering of the form $\{ U_J\} $
constructed above.
\end{lemma}

We remark that in a standard covering $\{ U_J\}$, the intersections of elements are again elements, since intersections of strips
are included as strips too (that was why we included the $H_i$ themselves).
We may now give a more explicit description of the folding maps.

\begin{corollary}
Suppose $E$ and $F$ are enclosures. Any folding map, that is to say a map to the associated sheaf $E\rightarrow \widetilde{F}$,
is given by taking a standard covering $\{ U_J\}$ and assigning for each $J\subset \{ H^+_i, H^-_i\}$
an affine map $a_J: U_J\rightarrow F$, subject to the condition that if $U_K\subset U_J$ then
$a_J |_{U_K} = a_K$. Two folding maps are the same if and only if they are the same pointwise, which is equivalent to saying
that they are the same on a common refinement of the two covers; a common refinement may be obtained by taking the unions of
the sequences of hyperplanes.
\end{corollary}

\begin{corollary}
A folding map between enclosures is finite-to-one.
\end{corollary}

We now consider a sheaf $F$ on $\Enc$. We say that it is {\em finitely generated} if there is a finite collection of
maps $E_i\rightarrow F$ from enclosures, such that the map of presheaves
$$
\coprod E_i \rightarrow F
$$
is surjective in the sheaf-theoretical sense, i.e. it induces a surjection of associated sheaves.

We say that $F$ is {\em finitely related} if, for any two maps from enclosures $E,E'\rightarrow F$, the
fiber product $E\times _FE'$ is finitely generated. We say that $F$ is {\em finitely presented} if it is finitely
related and finitely generated.

\begin{lemma}
If $E$ is an enclosure, then any subsheaf $F\subset \widetilde{E}$ is finitely related. More generally if
$\{ E_i\} _{i=1,\ldots , n}$ is a nonempty collection of enclosures, then any subsheaf
$$
F\subset \prod _i \widetilde{E}_i
$$
is finitely related.
\end{lemma}
\begin{proof}
We need to consider the fiber product $G\times _FG'$ for two maps from enclosures $G,G'\rightarrow F$.
These maps correspond to sequences $(\zeta _1,\ldots , \zeta _n)$ with $\zeta _i:G\rightarrow \widetilde{E_i}$,
resp.  $(\zeta '_1,\ldots , \zeta '_n)$ with $\zeta '_i:G'\rightarrow \widetilde{E_i}$.

Suppose $G=\bigcup G_J$ (resp. $G'=\bigcup G'_{J'}$) are (finite) coverings by enclosures. Suppose we can prove that
$G_J\times _FG'_{J'}$ are finitely generated. One can then conclude that $G\times _FG'$ is finitely generated.

Apply this to a common refinement of the coverings needed to define the folding maps $\zeta _i$, in the
above standard form. We conclude that it suffices to consider the case where $\zeta _i$ are affine maps.

Now in this case (and no longer using the notation for the coverings), the fiber product is expressed as
$$
G\times _FG' = \{ (x,x') \mbox{ s.t. } \zeta _i (x) = \zeta '_i(x') \mbox{ for } i=1,\ldots , n\} .
$$
This expression is somewhat heuristic as we are really talking about sheaves but it serves to indicate the proof.

Since $\zeta _1$ is an isomorphism we may assume that it is the identity and same for $\zeta '_1$.
Therefore, the first condition says that $x'=x$ and with this normalization, we may write
$$
G\times _FG' = \{  x\in A \mbox{ s.t. } x\in G\cap G', \;\; \zeta _i (x) = \zeta '_i(x) \mbox{ for } i=2,\ldots , n\} .
$$
The conditions $\zeta _i (x) = \zeta '_i(x)$ define Weyl hyperplanes or are always  true, so this represents
$G\times _FG' $ as an enclosure. This completes the proof.
\end{proof}

Note that $E\times E'$ will never be finitely generated as soon as ${\rm dim} (A) \geq 1$. Therefore, the
empty direct product which is to say the terminal object $\ast$ in sheaves on enclosures, is not finitely related.
The above proof used $n\geq 1$ in an essential way.

One should not confuse the terminal object $\ast$ with the enclosure $p$ when $p\in A$ is a point.
There are no maps $E\rightarrow \tilde{p}$
for $E$ different from a point or the empty set.

A {\em construction} is a finitely related sheaf on the site of enclosures.

\begin{theorem}
The category of constructions is closed under finite colimits, and fiber products.
\end{theorem}

We can define a topological space underlying a construction. If $F$ is a construction, let $F(p)$ denote the set of points,
that is to say the set of maps $p\rightarrow F$ where $p\in A$ is a point (recall from above that this is different from the
terminal sheaf $\ast$). Give $F(p)$ a topology as follows: for any enclosure $E$, $E(p)$ (which is equal to $\widetilde{E}(p)$)
has a topology as a subset of the affine space $A$. Then we say that a subset $U\subset F(p)$ is open if
its pullback to $E(p)$ is open, for any enclosure $E$ and any map $E\rightarrow F$.

\begin{conjecture}
If $F$ is a construction then the topological space $F(p)$ Hausdorff; furthermore it is a CW-complex.
\end{conjecture}

It might be necessary to add additional hypotheses on $F$ in order to insure that $F(p)$ is a CW complex.

\subsection{Spherical theory}

We would like to consider the local structure of a construction at a point. For this we need a spherical version of the above
theory. It seems like we probably don't need to consider ``enclosures'' in the spherical building but only ``sectors''.
A {\em sector} is a minimal closed chamber of a given dimension in the spherical complex associated to $A$.
The set of sectors is partially ordered by inclusion and the spherical Weyl group acts on it. It is a finite topological space,
in particular it has a structure of site (the only coverings of a sector must include that sector itself).
Let $\Sector$ be the category of sectors.

A {\em spherical construction} is a presheaf or equivalently sheaf on $\Sector$.

We have a map from $\Sector$ to the filters on $A$ located at any given point.

Suppose $F$ is a sheaf on $\Enc$ and $x\in F(p)$. We would like to associate the spherical construction $F_x$,
defined as follows: if $\sigma \in \Sector$ then $\sigma$ corresponds to a filter of enclosures,
that is to say a filtered category of enclosures $E$ with $p$ on the boundary of $E$ and whose local corner at $p$ looks like
$\sigma$. Call this category $\langle \sigma \rangle$. For $E\in \langle \sigma \rangle$,
consider the set $F(E)_x$ consisting of maps $E\rightarrow F$ such that $p$ maps to $x$.  Then we set
$$
F_x (\sigma ) := \lim _{\rightarrow , E\in \langle \sigma \rangle} F(E)_x .
$$
An element of $F_x(\sigma )$ therefore consists of a germ of map $E\rightarrow F$ sending $p$ to $x$, such that
the corner of $E$ at $p$ is $\sigma$. These germs are
up to equivalence that if two maps agree on a smaller enclosure also containing $p$ and having $\sigma$ as corner, then
the two germs are said to be equivalent.

\subsection{$\rr$-trees}
When the group is $SL_2$, the standard apartment is just $\rr$ and the enclosures are closed bounded
segments. The category of constructions gives a good point of view for the theory of $\rr$-trees.
For example, if $q$ is a quadratic differential defining a spectral multivalued differential
$\phi _i=\pm \sqrt{q}$ on a compact Riemann surface $X$, then the tree $T_{\varphi}$ of leaves of
the foliation $\Re \phi$ on $\widetilde{X}$ may be seen as a sheaf on $\Enc_{SL_2}$ as follows:
for a segment $E$ let $T^{\rm pre}_{\phi}(E)$ be the space of differentiable maps $E\rightarrow \widetilde{X}$
which are transverse to the foliation such that the pullback of $\Re\phi$ is the standard
differential $dx$ on $E\subset \rr$. These maps are taken
modulo the relation that two maps are the same if they map points of $E$ to the same leaves of the
foliation. Then $T_{\phi}$ is the associated sheaf. Thus $T_{\phi}(E)$ is the space of maps from $E$ to
the leaf space, which are represented on finitely many segments covering $E$ by differentiable maps
into $\widetilde{X}$.

\subsection{The $SL_3$ case}
We now specialize to the case of buildings for the group $SL_3$. The affine space is $A=\rr^2$. The spherical Weyl group
is the symmetric group acting through its irreducible $2$-dimensional representation.
Some enclosures were pictured in Figure \ref{someencs}.
There are three directions of reflection hyperplanes. These divide the vector space at the origin into six $2$-dimensional
sectors, acted upon transitively by the Weyl group. On the other hand, there are two orbits for the $1$-dimensional sectors, the
even vertices of the hexagon and the odd vertices. Therefore our category of sectors $\Sector$ is equivalent to the following category
$\Sector '$:
there are an object $\eta$, corresponding to the $2$-dimensional sectors, and two objects $\nu ^+$ and $\nu ^-$
corresponding to the $1$-dimensional sectors. We choose one of these denoted $\nu ^+$ which we say has
positive orientation. The morphisms are
$$
\nu ^- \rightarrow \eta \leftarrow \nu ^+.
$$
A spherical construction $H$ is a sheaf on $\Sector '$. This consists of three sets $H(\eta )$, $H(\nu ^+)$ and
$H(\nu ^-)$ with morphisms
$$
H(\nu ^- ) \leftarrow H(\eta ) \rightarrow H( \nu ^+) .
$$
Such a structure may be viewed as a graph whose edges are $H(\eta )$ with vertices grouped into the
positive ones $H(\nu ^+)$ and negative ones $H(\nu ^-)$. Each edge joins a positive vertex to a negative vertex.
A spherical construction is equivalent to such a graph.

If we have a construction $F$ for this Weyl group and if $x\in F(p)$ is a point then we obtain a spherical construction $F_x$ which
is a graph as above.

Following the simple characterization which was given
for example by Abramenko and Brown \cite{AbramenkoBrown}, we say that a spherical construction is a {\em spherical building} if any two vertices are at distance $\leq 3$, every pair of vertices
is contained in a hexagon,  and if there
are no loops of length $\leq 4$ (the length of a loop has to be even because of the parity property of edges).

A spherical construction is a {\em spherical pre-building} if it is connected, if every node is contained in at least one edge,
and if it has no loops of length $\leq 4$.
The construction of \cite{KNPS} gives a way of going from a spherical pre-building to a spherical building.

In a spherical pre-building, say that two nodes are {\em opposite} if they have opposite parity and are at distance
$3$. If it is spherical building this means that
they are opposite nodes of any
hexagon containing them.

A {\em segment} is a $1$-dimensional enclosure $S\subset A$. We note that a segment has a natural orientation. At any point $x\in S(p)$ in the interior, the spherical building $S_x$
has two elements, $S_{x,+}$ and $S_{x,-}$, not joined by any edge; they are of opposite parity
and the positive direction in $S(p)$ is defined to be the direction going towards $S_{x,+}$.
If $x$ is an endpoint then $S_x$ has one element, oriented positive or negative respectively
at the two endpoints of $S$.

Let us denote by $S^{t,+}$ the segment of length $t$ based at the origin, such that the parity of the single element of
$S^{t,+}_0$ is positive. We assume that these segments are all in the same line so that $S^{t,+}\subset S^{t',+}$ when
$t<t'$. Let $S^{t,-}$ denote the segment with the opposite orientation at the origin.

\begin{remark}
Suppose $F$ is a construction. If $\varphi : S^{t,+}\rightarrow F$ is a morphism then for $x\in S^{t,+}$, the image
of the positive (resp. negative) element of $S^{t,+}_x$ under $\varphi$ is denoted $\varphi _{x,+}$ (resp. $\varphi _{x,-}$)
and it is a positive (resp. negative) node in the spherical building $F_{\varphi (x)}$. Same for $S^{t,+-}$.
\end{remark}

\begin{definition}
\label{immersive}
Suppose $E$ is an enclosure, and $F$ a construction satisfying SPB-loc. We say that a map $f:E\rightarrow F$ is {\em immersive}
at a point $x\in E(p)$, if the map of spherical constructions
$E_x\rightarrow F_{f(x)}$ preserves distances. The map is {\em immersive} if it is immersive at all points of
$E(p)$.
\end{definition}

Here, the distances in $E_x$ are
calculated by considering $E_x\subset A_x$.
In particular, if $E$ is a segment then we consider the
distance between the two elements of $E_x$ to be $3$.

Thus a map from a segment $\varphi : S\rightarrow F$
is immersive at an interior point $x\in S$ if $\varphi _{x,+}$ and $\varphi _{x,-}$ are opposite in
the spherical pre-building $F_{\varphi (x)}$. We also use the terminology
{\em straight} for an immersive segment, and say that a map from a segment is {\em angular} otherwise.

Intuitively, a map $f: E\rightarrow F$ is
immersive if and only if $f(p):E(p)\rightarrow F(p)$ is locally injective.

\begin{lemma}
Suppose that the spherical constructions $F_x$ are at least spherical pre-buildings.
A map $\varphi $ is immersive at all but at most finitely many angular points. We say that $\varphi$ is {\em immersive}
if there are no angular points.
\end{lemma}

We now list some extension conditions for a construction $F$ (let us reiterate that we are working here in the $SL_3$ situation).

\noindent
{\bf SPB-loc} : that for any $x\in F(p)$ the spherical construction $F_x$ is a spherical pre-building.

\noindent
{\bf SB-loc} : for any $x\in F(p)$ the spherical construction $F_x$ is a spherical building.

\noindent
{\bf Ex-Seg} : let $t$ denote the other endpoint of $S^{t,+}$.
Assuming SPB-loc, for any immersive map $\varphi : S^{t,+}\rightarrow F$ and any element $\nu \in F_{\varphi (t)}$ opposite to
$\varphi ^-_t$, and for any $t'>t$ we ask that there exist an extension of $\varphi $ to an immersive map $\varphi ': [0,t']\rightarrow F$
such that $(\varphi ')^+_t= \nu$. Similarly for segments $S^{t,-}$ in the opposite orientation.

We next get to our main extension statement, for obtuse angles. Some notation will be needed first.

Let $P^{a,b, +}$ be parallelograms centered at the origin with side lengths $a$ and $b$, and positive orientation
of the two edges at the origin
(resp. $P^{a,b, -}$ with negative orientation). We may assume that the first edge is the segment
$S^{a,+}$, and denote the second edge by $\omega S^{b,+}$ (it is obtained by rotating by $120$ degrees).

Similarly let $T^a$ denote the triangle with edge length $a$. We consider both segments $S^{a,+}$ and
$>S^{a,-}$ to be edges of $T^a$ starting from the origin.

\noindent
{\bf Ex-Obt} : Assuming SPB-loc, suppose we are given maps $\varphi : S^{a,+}\rightarrow F$
and $\psi : \omega  S^{b,+}\rightarrow F$ such that $\varphi (0)=\psi (0)=x$. Suppose that
the distance between the elements $\varphi _{0,+}$ and $\psi _{0,+}$ in the spherical building $F_x$ is
$\geq 2$ (i.e. they are distinct). Suppose that the maps $\varphi$ and
$\psi$ are immersive. Then there exists an immersive map $\zeta : P^{a,b, +}\rightarrow F$
coinciding with the given maps $\varphi$ and $\psi$ on the edges. We also ask the same condition
with the other orientation, for $P^{a,b,-}$.

\noindent
{\bf Ex-Side} :
given a segment $\varphi : S\rightarrow F$ of length $a$ with $x=f(0)$ one of the endpoints, and an edge in the spherical
building $\nu \in F_x$,
$\varphi$ extends to an immersed triangle $T^a\rightarrow F$ such that $\nu$ is not in the image of $\varphi _0$.

\begin{lemma}
Suppose $F$ is a construction satisfying SPB-loc, Ex-Obt and Ex-Side. Then $F$ satisfies SB-loc.
\end{lemma}
\begin{proof}
Suppose $x\in F(p)$. By hypothesis $F_x$ is a spherical pre-building. We would like to show it is a spherical building.
From Ex-Obt, we get that any two nodes of the same parity in $F_x$ have distance $\leq 2$.
By hypothesis SPB-loc, the $F_x$ are connected graphs and all nodes are contained in edges. It follows
that any two nodes have distance $\leq 3$.
The condition Ex-Side implies that any node in any of the spherical
pre-buildings $F_x$, is contained in at least two distinct edges.
It now follows that any two nodes are contained in a hexagon.
Thus $F_x$ is a spherical building.
\end{proof}

\begin{lemma}
Suppose $F$ is a construction satisfying SB-loc, Ex-Obt and Ex-Seg.  Then it satisfies Ex-Side. Furthermore, the
sector of the triangle can be specified at either of the endpoints of the segment.
\end{lemma}

\begin{theorem}
\label{common-enclosure}
Suppose $F$ is a construction satisfying SB-loc, Ex-Seg, and Ex-Obt, such that $F(p)$ is connected. Suppose $x,y\in F(p)$. Then there exists
an immersive map from an enclosure $f:E\rightarrow F$ such that $x,y$ are in the image of $f(p)$.
\end{theorem}
\begin{proof} (Sketch)\,
Choose a path from $x$ to $y$, that is to say a continuous map from $[0,1]$ to the
topological space $F(p)$. We assume a topological result saying that the path may be covered by finitely many intervals which
map into single enclosures. From this, we may assume that the path has the following form: it goes through a series of triangles
which are immersive maps
$\varphi _i: T^a\rightarrow F$, such that $\varphi _i$ and $\varphi _{i+1}$ share a common edge. Notice here that we may assume
that all the triangles have the same edge length. We have $x$ in the image of $\varphi _1$ and $y$ in the image of $\varphi _k$.

We obtain in this way a strip of triangles. We next note that there are moves using Ex-Obt which allow us to assume that
there are not consecutive turns in the same direction along an edge.

From this condition, we may then extend using Ex-Obt and Ex-Side (which follows from Ex-Seg) to turn this strip of triangles, into a
single immersive map from an enclosure.
\end{proof}

We let $A$ denote the sheaf represented by the affine space of the same notation. More precisely,
for an enclosure $E$ a map $E\rightarrow A$ consists of a factorization $E\rightarrow E'\rightarrow A$ where
$E'\subset A$ is an enclosure in its standard position, and $E\rightarrow E'$ is a folding map (that is, a map to
$\widetilde{E}'$). Thus $A$ is a direct limit of its enclosures. In particular
we know what it means to have an immersive map from $A$ to somewhere, the map should be immersive on each of the
enclosures in $A$.  If $F$ is a construction, an {\em apartment} is an immersive map $A\rightarrow F$.

\begin{proposition}
Suppose $F$ is a construction satisfying SB-loc, Ex-Obt and Ex-Seg. Then any immersive map from an enclosure
$E\rightarrow F$ extends to an apartment $A\rightarrow F$.
\end{proposition}

\begin{corollary}
\label{common-apartment}
If $F$ is a construction satisfying SB-loc, Ex-Obt and Ex-Seg, such that $F(p)$ is connected,
 then any two points $x,y\in F(p)$ are contained in a common
apartment.
\end{corollary}
\begin{proof}
Combine the previous proposition with Theorem \ref{common-enclosure}.
\end{proof}

\begin{proposition}
Suppose $F$ is a construction satisfying SB-loc. Then an immersive map $E\rightarrow F$ from an enclosure, is injective.
The same is true for an apartment $A\rightarrow F$.
\end{proposition}

\begin{definition}
\label{def-building}
Say that $F$ is a {\em building} if it is a construction with $F(p)$ connected and simply connected,
satisfying SB-loc, Ex-Obt and Ex-Seg.
\end{definition}

\begin{conjecture}
\label{itsabuilding}
If $F$ is a building,
then the topological space $F(p)$ has a natural
structure of $\rr$-building modeled on the affine space $A$.
\end{conjecture}

\subsection{Pre-buildings}

\begin{definition}
\label{pb-def}
A {\em pre-building} is a construction $F$ which satisfies SPB-loc such that $F(p)$ is simply connected.
\end{definition}

Suppose $F$ is a pre-building. Recall that a segment $s : S\rightarrow F$ is straight
if it is immersive, that is for any $y\in S(p)$ which is not an endpoint, the two directions of $S_y$
map to two vertices of $F_{s(y)}$ which are separated by three edges.

\begin{definition}
\label{pb-std}
A map $V^{a,b}\rightarrow F$ to a pre-building is {\em standard} if the two edges are straight and
their two directions at the origin are not the same.
\end{definition}

We would now like to investigate the process of adding a parallelogram $P^{a,b}$ to
a pre-building $F$ along a standard inclusion $V^{a,b}\rightarrow F$ to get a new
pre-building. The problem is that a similar parallelogram, or a part of it, might already be there.
For example if we add $P^{a,b}$ to itself by taking a pushout along $V^{a,b}$ then we generate
a fourfold point at the origin.

Thus the need for a process which we call {\em folding in}, of joining up the new $P^{a,b}$ with
any part of it that might have already been there. This will be used in the small-object argument
below as well as in our general modification procedure later.

Suppose $F$ is a pre-building and $d:V^{a,b}\rightarrow F$ is a standard inclusion. Define $R_d\subset
P^{a,b}$ to be the union of all of the $P^{c,d}\subset P^{a,b}$ based at the origin, such that
there exists a map on the bottom making the following diagram commute:
$$
\begin{array}{ccc}
V^{c,d} & \hookrightarrow & V^{a,b} \\
\downarrow & & \downarrow \\
P^{c,d}& \rightarrow & F.
\end{array}
$$
One can show that the map, if it exists, is unique. We obtain a subsheaf $R\subset P^{a,b}$
and there is a map $R\rightarrow F$ combining all of the previously mentioned maps. The {\em folding-in}
of $P^{a,b}$ along $d$ is the pushout
$$
G:= F\cup ^{R}P^{a,b}.
$$
The relation $R$ is designed so that
$G$ again satisfies SPB-loc.
Since $R$ was defined as a very general union of things,
there remains open the somewhat subtle technical issue, which we haven't treated,
of showing that $G$ is a construction.

The folding-in $G$ accepts a map $F\cup ^{V^{a,b}}:P^{a,b}\rightarrow G$ and we have the following
universal property:

\begin{lemma}
\label{foldingin}
The folding-in $G$ is a pre-building.
Suppose $B$ is a pre-building and $f:F\rightarrow B$ is a map such that the image of $d$ is
again a standard inclusion into $B$ (as will be the case for example if $f$ is non-folding).
Suppose that $V^{a,b}\rightarrow B$ completes to a map $P^{a,b}\rightarrow B$. Then these
factor through a unique map $G\rightarrow B$.
\end{lemma}

We also conjecture a versality property for maps to buildings: if $F\rightarrow B$ is any map
to a building then there exists an extension to $G\rightarrow B$, however this extension
might not be unique. Because of the non-uniqueness on the interior pieces making up $R$,
this conjecture requires a study of convergence issues and the statement might need to be modified.

\subsection{The small-object argument}
\label{soa}

The conditions Ex-Seg, Ex-Obt and Ex-Side are extension conditions of the following form: we have a certain arrangement
$R$ consisting of some points or segments, and a map $R\rightarrow E$ to an enclosure putting $R$ on the boundary of $E$.
Then the conditions state that for any immersive map $R\rightarrow F$ there exists an extension to an immersive map $E\rightarrow F$.

These conditions may be ensured using the small object argument.
Given $F$ we may define ${\rm Ex}(F)$ to be
the construction obtained by a successive infinite family of pushouts along the inclusions $R\rightarrow E$,
using the folding-in process described in the previous subsection for ${\rm Ex-Obt}$.

\begin{theorem}
\label{small-object}
If $F$ is a pre-building, and if we form $F':={\rm Ex}(F)$ by iterating up to the first infinite
ordinal $\omega$, then $F'$ is a building.
\end{theorem}
\begin{proof} (Sketch) We first note that
$F'$ is a construction, and $F'(p)$ is simply connected. Also  $F'$ again satisfies SBP-loc.
By the small object argument,
 $F'$ satisfies the extension properties Ex-Seg, Ex-Obt and Ex-Side.
For Ex-Obt, the folding-in process only adds parallelograms along
standard inclusions, but this turns out to be good enough to get it in general.
Hence, $F'$ satisfies SB-loc, so it is a building. In particular any two points of $F'(p)$ are contained in a common apartment.
\end{proof}

Conjecture \ref{itsabuilding} then says that $F'(p)$ has a natural structure of $\rr$-building.

\begin{conjecture}
Suppose $F$ is a pre-building, and let $F'$ be the building obtained by the small-object argument.
Then it is versal: for any other map to a building $F\rightarrow B$ there exists a factorization
through $F'\rightarrow B$.
\end{conjecture}

The difficulty is that we needed to use the folding-in construction to construct $F'$,
and the versality property for the folding-in construction is not easy to see because of
the possibly  infinite nature of the relations $R$ coupled with non-uniqueness of extensions
to $P^{a,b}$ for non-standard maps $V^{a,b}\rightarrow B$. The small-object construction may
need to be modified in order to get to this versality conjecture.

\section{An initial construction}
\label{sec-initial}

The next step is to to relate constructions to $\phi$-harmonic maps. Consider a Riemann surface $X$ together
with a spectral curve, defining a multivalued differential $\phi$. Recall that locally on $X^{\ast}$,
$\phi$ may be written as
$(\phi _1,\ldots , \phi _n)$ with $\phi _i$ holomorphic, and $\sum \phi _i=0$. In the present paper we consider
spectral curves for $SL_3$ so $\phi = (\phi _1,\phi _2, \phi _3)$. The {\em foliations} are defined by $\Re \phi _{ij}=0$.
In our $SL_3$ case there are three foliation lines going through each point of $X^{\ast}$.

For any point $x\in X^{\ast}$ there exists a neighborhood $U_x$, with the property \cite{KNPS} that
$U_x$ has to map to a single apartment under any $\phi$-harmonic map.
Thus locally, a $\phi$-harmonic
map to a building factors through the map
$$
h^x:U_x\rightarrow A
$$
to the standard apartment given by integrating the forms $\Re \phi _i$ with basepoint $x$.
The foliation lines defined by $\Re \phi _{ij}=0$ are just the preimages in $U_x$ of the
reflection hyperplanes of $A$. The preimage of a small enclosure near the origin in $A$
will therefore be a domain in $U_x$ whose boundary is composed of foliation lines for the
three foliations. We may shrink $U_x$ so that its closure is itself such a domain.

\subsection{The $\Omega _{pq}$ argument}

Recall briefly the proof of \cite{KNPS}
showing existence of such a neighborhood $U_x$. A path $\gamma : [0,1]\rightarrow
X$ is {\em noncritical} if the differentials $\gamma ^{\ast}(\Re \phi _i)$ remain in the same
order all along the path. For any $\phi$-harmonic
map to a building $h:X\rightarrow B$, the image of a noncritical
path has to be contained in a single apartment. One shows this using the property of a building that says
two opposite sectors based at any point are contained in a single apartment.

Suppose $p,q\in X$ are joined by some noncritical path. Let $\Omega _{pq}$ denote the subset swept out
by all of these paths, in other words it is the set of all points
$y\in X$ such that there exists a noncritical path $\gamma$ with $\gamma (0)=p$, $\gamma (1)=q$
and $\gamma (s)=y$ for some $s\in [0,1]$. We claim that $\Omega _{pq}$ maps into a single apartment.
Suppose $y,y'\in \Omega _{pq}$ and let $\gamma$ and $\gamma '$ denote the corresponding paths.
We have apartments $A,A'\subset B$ such that $h\circ \gamma$ goes into $A$ and $h\circ \gamma '$ goes
into $A'$. Now $A\cap A'$ is a Finsler-convex subset in either of the two apartments,
and it contains $p$ and $q$. In particular it contains the parallelogram with opposite
endpoints $p$ and $q$.  It follows that it contains the images of the two paths, so
$h(y)$ and $h(y')$ are both in $A$. Letting $y'$ vary we get that $h(\Omega _{pq})\subset A$.

The $\phi _i$ admit a uniform determination over
$\Omega _{pq}$, and
the map $\Omega _{pq}\rightarrow A$ is determined just by integrating the real one-forms $\Re \phi _i$.

Now if $x\in X^{\ast}$, we may find nearby points $p,q\in X$ such that $x$ is
in the interior of $\Omega _{pq}$. This is easy to see away from the caustic lines. The caustic
lines are transverse to all of the differentials so if $x$ lies on
a caustic one can choose $p$ and $q$ on the
caustic itself, as will show up in our pictures later. This still works
at an intersection of caustics too.

Our neighborhood $U_x$ is now chosen to be any neighborhood of $x$ contained in some $\Omega _{pq}$.
This construction is uniform, independent of the harmonic map $h$.

\subsection{Caustics}
A point $x\in X^{\ast}$ is {\em on a caustic} if the three foliation lines are tangent. This is equivalent to saying that the
three points $\phi _i(x)\in T^{\ast}X_x$ are aligned, i.e. they are on a single real segment. The {\em caustics} are the
curves of points in $X$ satisfying this condition. We include also the branch points in the caustics.

There is a single caustic
coming out of each ordinary branch point.

The caustics play a fundamental role in the geometry of harmonic maps to buildings, specially in the $SL_3$ case.

Let $C\subset X$ denote the union of the caustic curves. Then, for any
$x\in X-C$ the map $h^x:U_x\rightarrow A$ is etale at $x$.
Hence, if $h:X\rightarrow B$ is a $\phi$-harmonic map
to a building, $h$ is etale onto the local apartments outside of $C$.
 In other words, the
local integrals of any two of the differentials provide local coordinate systems on $X-C$.
On the other hand, $h$ folds $X$ along $C$.

We may also view this as determining a flat Riemannian metric on $X-C$, pulled back from the standard Weyl-invariant metric
by the local maps $h^x$ to the standard apartment obtained by integrating $\Re \phi _i$. The metric has a distributional
curvature concentrated along $C$. From the pictures it seems that the curvature is everywhere negative,
and from our process we shall see that the total amount of curvature along a single caustic joining two branch points
gives an excess angle of $120^{\circ}$.

\subsection{Non-caustic points}
\label{noncaustic}
If $x\in X-C$ is a non-caustic point, then we may assume that the neighborhood $U_x$ maps isomorphically
to the interior of an enclosure in the standard apartment. The enclosure $E_x$
could be chosen as a standard hexagon or perhaps
a standard parallelogram for example.

We have chosen $U_x$ such that for any harmonic $\phi$-map to a building $h:X\rightarrow B$,
the map $U_x\rightarrow B$ factors through an apartment $A\subset B$
via the map $h^x:U_x\rightarrow A$ given by integrating the $\Re \phi _i$. The map
$h^x$ factors through an affine (non-folding) map $E_x\rightarrow A$. Altogether we get a factorization
\begin{equation}
\label{thefact}
U_x\rightarrow E_x \rightarrow B
\end{equation}
and the map $E_x\rightarrow B$ doesn't fold along any edges passing through $x$.

When $x$ is a branch point or on a caustic, we can still get a local factorization of the form
\eqref{thefact} through a construction $E_x$, as will be discussed next.

\subsection{Smooth points of caustics}
Suppose $x\in C$ is a point in the smooth locus of a caustic curve. Then the local integration map folds
every neighborhood of $x$ in two along $C$. The image of $U_x$ by $h^x$ is folded along $C$.
By intersecting with a smaller enclosure containing the image of $x$, we may assume that $U_x$ has the property that
there is an enclosure $E_x$ with $h^x: U_x\rightarrow E_x^{\circ}$ being a proper $2$ to $1$ covering folded along $C$.
We may also assume that $E_x$ is the convex hull of the closed $h^x(\overline{U}_x)$.

For the generic situation, there are two other types of points that need to be considered: the intersections of caustics, and
the branch points.

\subsection{Branch points}
For the branch points, locally two of the differentials say $\phi _1$ and $\phi _2$ come together, and the third one could
be considered as independent. Therefore, a $\phi$-harmonic map looks locally like the projection to the tree of leaves of
a quadratic differential, crossed with a real segment. It means that we are forced to consider a singular construction
rather than an enclosure. This singular construction still denoted $E_x$
may be taken for example as the union of three half-hexagons joined along
their diameters.

\begin{minipage}[h]{\textwidth}
\centering
\includegraphics[width=0.65\textwidth]{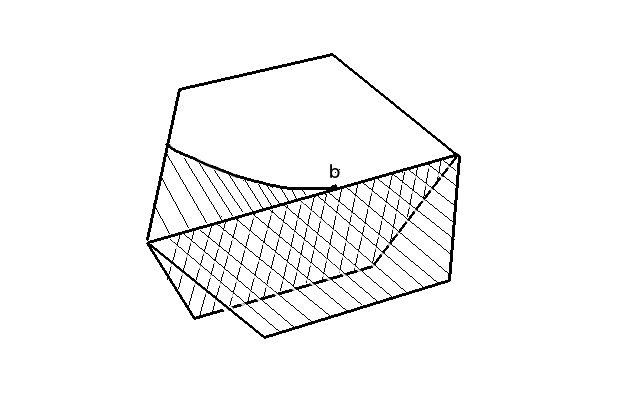}
\captionof{figure}{Three half-hexagons}
\label{branch}
\end{minipage}

\bigskip

If $x=b$ is a branch point then we may choose the neighborhood $U_x$ together with a map $h^x : \overline{U}_x\rightarrow
E_x$ such that $E_x$ is the convex hull of the image. The image of $\overline{U}_x$, shaded in above,
looks locally like the one that we saw in the BNR example
\cite{KNPS}.

\subsection{Crossing of caustics}

The case where $x$ is a crossing point of two caustics is new, not appearing in the BNR example. We have therefore looked fairly closely
at this situation in one of the next basic examples. A spectral network with the two crossing caustics will be
shown later as Figure \ref{a2sn} of Section \ref{sec-a2}.

The $\Omega _{pq}$ argument works also here, so we have a neighborhood $U_x$ of $x$ which has to map to a single
apartment in any $\phi$-harmonic map to a building. The two caustics divide $U_x$ into
four sectors. The map $h^x: U_x \rightarrow A$ is $1$ to $1$ in the interior of two of the opposing sectors; it is $3$ to $1$
in the interior of the other two opposing sectors; it is $2$ to $1$ along the caustics and $1$ to $1$ at $x$.
This is emphasized by including the images of two circles in the picture shown in Figure \ref{hexagon}.

We choose $\overline{U}_x$ as a hexagonal shaped region shown in Figure \ref{hexagon} on the left. The
image $E_x$ in a standard apartment is a hexagon-shaped enclosure, shown on the right in  Figure \ref{hexagon}.
As said above, the map is $3:1$ over the thin middle regions between the two caustics.
The map $h_x: \overline{U}_x \rightarrow E_x$ is proper and
$\overline{U}_x$ is the inverse image of $E_x$.

\begin{figure}
\centering
\includegraphics[width=1.0\textwidth]{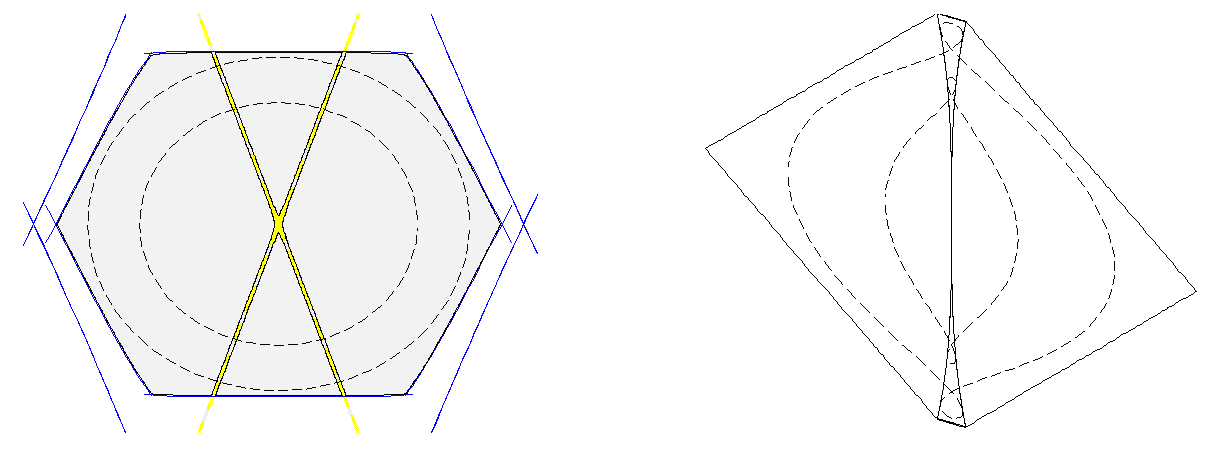}
\caption{A neighborhood of the crossing}
\label{hexagon}
\end{figure}

\subsection{The initial construction}

We now put together the above neighborhoods to get a good covering of $X$.

\begin{theorem}
There exists a finite covering of $X$ by open sets $U_i$, and constructions $E_i$ as considered above
(either an enclosure or the union of three half-hexagons),
together with maps $h_i : \overline{U_i} \rightarrow E_i$ such that for any harmonic $\phi$-map
to a building $h: X\rightarrow B$, there is an isometric embedding $E_i\subset B$ such that $h$
factors through $h_i$ and indeed $\overline{U}_i$ is a connected component of $h^{-1}(E_i)$.
\end{theorem}

For the intersections $U_{ij}:= U_i\cap U_j$, we obtain constructions $E_{ij}$ with inclusions to both $E_i$ and
$E_j$, such that $E_{ij}$ is the convex hull of $h_i(U_{ij})$ in $U_i$ (resp. the convex hull of $h_j(U_{ij})$ in $U_j$).

Define
$$
Z:=\frac{\bigcup _i E_i} {\sim }
$$
where the relation $\sim$ is obtained by identifying the $E_{ij}\subset E_i$ with $E_{ij}\subset E_j$.

\begin{theorem}
\label{initialcons}
This defines a construction $Z$. We have a $\phi$-harmonic map $h_Z: X\rightarrow Z$ and it has the following
universal property: for any $\phi$-harmonic map to a building $h: X\rightarrow B$, there is a factorization
$h = h'\circ h_Z$ for a unique map of constructions $h': Z\rightarrow B$.
\end{theorem}

The construction $Z$ is called our {\em initial construction}.
It is not a pre-building because it will not, in general,
have the required nonpositive curvature property. For example the initial construction
described in Section \ref{BNRrev} will have fourfold points $a_1$ and $a_3$.

\begin{lemma}
\label{no2}
We can nonetheless insure that the local spherical constructions of $Z$ don't have cycles of length $2$.
\end{lemma}

In the next section we look at how to modify the initial construction in order to remove the positively curved points,
that is to say the points where the spherical building has a cycle of length $4$.

\section{Modifying constructions}
\label{sec-modifying}

In this section we consider how to go from the initial construction to a pre-building by a sequence of modification
steps. The reader is referred to Section \ref{BNRrev} for an illustration of the various operations
to be described here.

These operations are very similar in spirit to foldings and trimmings in the theory of Stallings graphs
\cite{Stallings, KapovichMyasnikov, ParzanchevskiPuder}, and the two-manifold construction that comes out
at the end should be considered as a ``core''.

\subsection{Scaffolding}

In order to keep track of what kind of folding happens, we first look at
some extra information that can be attached to a construction. In this subsection we
remain as usual in the $SL_3$ case.

Let $F$ be a construction. An {\em edge germ} of $F$  is defined to be a quadruple $(x,v,a,b)$ where $x\in F(p)$ and
$v$ is a vertex in the spherical construction $F_x$, and $a$ and $b$ are edges in $F_x$ sharing
$v$ as a common endpoint.

There is a change of dimension when passing from the spherical building to the construction itself,
so the vertex $v$ corresponds to a germ of $1$-dimensional segment based at $x$, and the
edges $a$ and $b$ correspond to germs of $2$-dimensional sectors based at $x$ that are
separated by the segment.

Suppose $f:F\rightarrow G$ is a map of constructions. If $(x,v,a,b)$ is an edge germ of $F$, then
the images $f(a)$ and $f(b)$ are edges in $G_{f(x)}$ sharing the
vertex $f(v)$. We say that
$f$ {\em folds} along $(x,v,a,b)$ if $f(a)$ and $f(b)$ coincide. We say that $f$ {\em opens} along $(x,v,a,b)$ if
it doesn't fold.

Let ${\bf EG} (F)$ denote the set of edge germs of $F$.
A {\em scaffolding} of a construction $F$ is a pair $\sigma = (\sigma ^o ,\sigma ^f)$ such that $\sigma ^o $ and $\sigma ^f$ are disjoint
subsets of ${\bf EG} (F)$. The first set $\sigma ^o $ is said to be the set of edge germs which are marked ``open'',
and the second set $\sigma ^f$ is said to be the set of edge germs which are marked ``fold''.

If $f:F\rightarrow G$ is a map of constructions, and $\sigma$ is a scaffolding of $F$, we say that
$f$ is {\em compatible with $\sigma$} if $f$ folds along the edge germs in $\sigma ^f$ and opens along the edge germs
in $\sigma ^o$.

Implicit in this terminology is that $G$ was provided with a fully open scaffolding (such as will
usually be the case for a pre-building). More
generally a map between constructions both provided with scaffoldings is compatible if it maps the open
edge germs in $F$ to open ones in $G$, and for edge germs marked ``fold'' in $F$ it either folds them
or else maps them into edge germs marked ``fold'' for $G$.

\begin{definition}
A scaffolding is {\em full} if $\sigma ^o \cup \sigma ^f = {\bf EG}(F)$. A scaffolding is {\em coherent} if there exists
a building $B$ and a map $h:F\rightarrow B$ compatible with $\sigma$.
\end{definition}

It should be possible to replace the definition of coherence with an explicit list of required properties, but we don't do that here.
Recall that we will be working under the assumption of existence of some $\phi$-harmonic map so the above definition is
adequate for our purposes.

One of the main properties following from coherence is a propagation property. A neighborhood of a hexagonal point cannot
be folded in an arbitrary way. One may list the possibilities, the main one being just folding in two; note however that
there is another interesting case of three fold lines alternating with open lines. Certain cases are ruled out and
we may conclude the following property:

\begin{lemma}
If $\sigma$ is coherent, and $x$ is a hexagonal point, then if two adjacent edge germs at $x$ are in $\sigma ^o$ it
follows that the two opposite edge germs are also in $\sigma ^o$.
\end{lemma}

This will mainly be used to propagate the open or fold edge germs along edges which are ``straight'' in the following sense. This definition coincides with Definition \ref{immersive} for a pre-building
provided with the fully open scaffolding.

\begin{definition}
\label{straight}
Suppose $S$ is a segment and $\varphi \in F(S)$. Suppose $\sigma$ is
a scaffolding for $F$. We say that $\varphi $ is {\em straight} if, at any point $x$ in the interior of $\varphi$ (that is to say
$x\in F(p)$ is the image of a point $a\in S(p)$ which is not an endpoint), the forward and backward directions
along $\varphi $ in the spherical construction $F_x$ are separated by three edges $a,b,c$ such that the
edge germs $ab$ and $bc$ are in $\sigma ^o$. Here the edge germ denoted $ab$ corresponds to the vertex
separating the edges $a$ and $b$ in the spherical construction.
\end{definition}

\begin{corollary}
\label{straightcoh}
Suppose $\varphi$ is a straight edge and $\sigma$ a coherent scaffolding.
Then the marking of both forward and backward edge germs of $\varphi$
is the same all along $\varphi$.
\end{corollary}

Straight edges are mapped to straight edges:

\begin{lemma}
Suppose $F$ is a construction with scaffolding $\sigma$ and $\varphi \in F(S)$ is an edge. If
$h:F\rightarrow B$ is any map to a building compatible with $\sigma$, and if $\varphi$ is straight with respect to $\sigma$,
then the image of $\varphi$ is contained in a single apartment of $B$ and is a straight line segment
in that or any other apartment.
\end{lemma}

Scaffolding is used to keep track of the additional information which comes from a $\phi$-harmonic
map, namely that small open neighborhoods in $X^{\ast}$ map without folding to apartments in
a building. For trees, this was illustrated in Figures \ref{littlenbds}, \ref{cornersegs} and \ref{cseg}.

\begin{proposition}
\label{initialscaff}
The initial construction $Z$ of Theorem \ref{initialcons} is provided with a coherent full scaffolding
$\sigma_Z$, such that if $x\in X-C$
is a non-caustic point, then all edge germs in the hexagon in $Z_{h_Z(x)}$ image of the local spherical construction of $E_x$ at the origin,
are in $\sigma _Z^o$. If $B$ is a building, there is a one-to-one correspondence between harmonic $\phi$-maps
$X\rightarrow B$ and construction maps $Z\rightarrow B$ compatible with $\sigma _Z$.
\end{proposition}
\begin{proof} (Sketch)
All edge germs based at any point $x\in X-C$ are in $\sigma ^o$ by the remark at the end of
Section \ref{noncaustic}. In the regions near caustics, intersections of caustics or branch points,
the local constructions $E_x$ map into any building without folding, so all of their edge germs
are in $\sigma ^o$. The remaining edges are always straight segments attached to fourfold points, and at these
fourfold points two of the edges are already marked ``open'' so the other two must be marked ``fold''.
This is illustrated in the BNR example in Section \ref{BNRrev}.  The scaffolding is coherent because
we are assuming that there exists at least one $\phi$-harmonic map
such as can be obtained from Parreau's theory
\cite{Parreau} \cite{KNPS}.
\end{proof}

\subsection{Cutting out}

We now describe an important operation. Suppose $(F,\sigma )$ is a construction with coherent scaffolding.
Suppose $P^{a,b}$ is a parallelogram (with either orientation), and denote by
$V^{a,b}$ the union of two segments based at the origin
on the boundary of $P^{a,b}$. Suppose that $F$ may be written
as a pushout
$$
F= G\cup ^{V^{a,b}}P^{a,b}
$$
along a map $g:V^{a,b}\rightarrow G$. Let $\sigma _G$ be the induced scaffolding of $G$. Suppose that
the two segments making up $V^{a,b}$, of lengths $a$ and $b$ respectively, are straight
in $G$ with respect to $\sigma _G$.

\begin{theorem}
\label{th-trim}
In the above situation,  if $B$ is a building, then any map $h_G:G\rightarrow B$
compatible with $\sigma _G$ extends uniquely to a map $h:F\rightarrow B$ compatible with $\sigma$.
Furthermore, under this map the image of the parallelogram $P^{a,b}$ is isometrically embedded in $B$ and
doesn't fold any edge germs of ${\bf EG}(P^{a,b})$.
\end{theorem}

In the above situation, we refer to $G$ as being obtained from $F$ by {\em cutting out} the image of the
parallelogram $P^{a,b}$. Notice that all of the edge germs in ${\bf EG}(P^{a,b})\subset {\bf EG}(F)$
are either marked ``open'', or not  marked. We may extend the scaffolding $\sigma$ to one $\sigma _1$
with $\sigma _1^f=\sigma ^f$ and $\sigma _1^o = \sigma ^o\cup {\bf EG}(P^{a,b})$,
and the theorem implies that any map from $F$ to a building compatible with $\sigma$ must also be
compatible with $\sigma _1$. If $\sigma$ is full then of course coherence implies that $\sigma _1=\sigma$.

Note also that if $\sigma$ is full then $\sigma _G$ is full.

We say that $(F,\sigma )$ is {\em trimmed} if it is not possible to cut out any parallelograms in the above way.

\begin{theorem}
\label{initialtrim}
The initial construction $Z$ of Theorem \ref{initialcons}, provided with the scaffolding $\sigma _Z$ of
Proposition \ref{initialscaff}, may be trimmed. The result is a two-manifold construction $Z_0$ again provided
with a coherent full scaffolding $\sigma _0$. If $B$ is a building,
there is a one-to-one correspondence between:
\begin{itemize}

\item Harmonic $\phi$-maps $X\rightarrow B$;

\item Maps $Z\rightarrow B$ compatible with $\sigma _Z$; and

\item Maps $Z_0\rightarrow B$ compatible with $\sigma _0$.

\end{itemize}
\end{theorem}

\subsection{Pasting together}
\label{sec-pasting}

The construction $Z_0$ now obtained will have, in general, many fourfold points of positive curvature. These cannot appear
in the image of a map $Z_0\rightarrow B$, so some folding must occur. The direction of the folds is determined by the
scaffolding, and indeed that was the reason for introducing the notion of scaffolding: without it, there is not
{\em a priori} any preferred way of determining the direction to fold a fourfold point. However, once the direction is
specified, we may proceed to glue together some further pieces of the construction according to the required folding.
This is the ``pasting together'' process.

Put
$$
W^{a,b}:= P^{a,b}\cup ^{V^{a,b}} P^{a,b}.
$$
We have a projection $\pi : W^{a,b}\rightarrow P^{a,b}$ given by the identity on each of the two pieces.

Here we may use either of the two possible orientations, that doesn't need to be specified but of course it should be the
same for both parallelograms.

The local spherical construction
of $W^{a,b}$  at the origin $0\in W^{a,b}$ is a graph with a single cycle of length $4$, in other words
the origin is a fourfold point.
In particular, if $W^{a,b}\rightarrow B$ is any map to a building, two of the four edge germs at $0$ must be folded.
There is a choice here: at least two opposite edge germs must be folded, but the other two could either be opened or
also folded. We are interested in the case when the originally given parallelograms are not folded.
Let $v_1,v_2$ be the edge germs at the origin in $W^{a,b}$ which are in the middle of the spherical constructions of the
two pieces $P^{a,b}$ (the spherical construction of $P^{a,b}$ at the origin is a graph with two adjacent edges and
three vertices, the middle edge germ corresponds to the middle vertex).

\begin{proposition}
\label{patchprop}
Suppose we are given a coherent scaffolding $\sigma _W$ of $W^{a,b}$. We assume that $v_1$ and $v_2$ are in $\sigma _W^o$,
and that the two edges comprising $V^{a,b}$ are straight. Then any map $h_W:W^{a,b}\rightarrow B$ to a building,
compatible with $\sigma _W$, factors through the projection
$$
W^{a,b}\stackrel{\pi}{\rightarrow} P^{a,b}\stackrel{h_P}{\rightarrow} B
$$
via a map $h_P$ which is an isometric embedding of the parallelogram into a single apartment of $B$.
\end{proposition}

Let $\sigma _{W,{\rm max}}$ be the full scaffolding defined by taking all edge germs of each $P^{a,b}$
to be open, and letting the edge germs along the edges of $V^{a,b}$ be folded. A corollary of the proposition
is that any coherent scaffolding of $W^{a,b}$ satisfying the hypotheses
has to be contained in $\sigma _{W,{\rm max}}$.

Now suppose $F$ is a construction with coherent scaffolding $\sigma$. Suppose we have an inclusion
$$
W^{a,b}\hookrightarrow F
$$
and suppose that the induced scaffolding $\sigma _W$ of $W^{a,b}$ satisfies the hypotheses of the proposition,
namely the angle between $v_1$ and $v_2$ is open and the edges of $V^{a,b}$ are straight. Define the quotient
$$
\widetilde{F} := F\cup ^{W^{a,b}} P^{a,b}
$$
which amounts to {\em pasting together} the two parallelograms which make up $W^{a,b}$.

\begin{corollary}
\label{patchcor}
Under the above hypotheses, any map to a building $h:F\rightarrow B$ compatible with $\sigma$ factors
through a map
$$
\widetilde{F}\rightarrow B.
$$
\end{corollary}

This operation  may be combined with the cutting-out operation. Let $\overline{V}^{a,b}$ denote the opposite
copy of $V^{a,b}$ inside $P^{a,b}$. We may write
$$
\partial W^{a,b} = \overline{V}^{a,b} \cup ^{p\sqcup p}  \overline{V}^{a,b},
$$
it is a rectangular one-dimensional construction formed from two edges of length $a$ and two edges of length $b$
(again as usual, we should chose one of the two possible orientations for everything). We may define a projection
$$
\partial W^{a,b} \rightarrow V^{a,b}
$$
as the identity on the two pieces. The combination of pasting together and cutting out may be summarized in the
following theorem.

\begin{theorem}
\label{patchcut}
Suppose $(F,\sigma )$ is a construction with coherent scaffolding. Suppose that $F$ can be written
$$
F= G\cup ^{\partial W^{a,b}}W^{a,b} .
$$
Suppose the scaffolding $\sigma _W$ induced by $\sigma$ on $W^{a,b}$ satisfies the hypotheses of Proposition
\ref{patchprop}. Define
$$
\widehat{G}:= G\cup ^{\partial W^{a,b}}V^{a,b} .
$$
We may write
$$
\widetilde{F}= \widehat{G}\cup ^{V^{a,b}}P^{a,b}.
$$
Furthermore, $\sigma$ induces scaffoldings $\sigma _G$ on $G$ and $\sigma _{\widehat{G}}$ on $\widehat{G}$.
Let $\sigma _{\widetilde{F}}$ be the scaffolding of $\widetilde{F}$ obtained from
$\sigma _{\widehat{G}}$ by declaring all edge germs of $P^{a,b}$ to be open.

Assume that $V^{a,b}\subset \widehat{G}$ is standard with respect to the
scaffolding $\sigma _{\widehat{G}}$. Then there is a one-to-one correspondence between the sets
$$
\left\{
\mbox{maps } h': \widehat{G}\rightarrow B \mbox{ to a building compatible with } \sigma _{\widehat{G}}
\right\}
$$
and
$$
\left\{
\mbox{maps } h: F\rightarrow B \mbox{ to a building compatible with } \sigma ,
\right\}
$$
both being the same as
$$
\left\{
\mbox{maps } h: \widetilde{F}\rightarrow B \mbox{ to a building compatible with } \sigma _{\widetilde{F}}
\right\}
$$
via the natural maps
$$
\widehat{G} \rightarrow \widetilde{F} \leftarrow F .
$$
\end{theorem}

The scaffolding $\sigma _{\widehat{G}}$ is never full, because it doesn't say what to do along the new edges
which have been introduced by glueing together the appropriate edges in $\partial W^{a,b}$. This is
the main problem for iterating the construction, and it eventually leads to the need to look for BPS states.

\subsection{Two-manifold constructions}

A {\em two-manifold construction} is a construction $F$ such that the topological space $F(p)$ is a $2$-dimensional
manifold.

\begin{lemma}
A construction $F$ is a two-manifold if and only if its local spherical constructions $F_x$ are polygons.
\end{lemma}

The polygons have an even number of edges because of the orientations of enclosures.

Suppose $F$ is a two-manifold construction. A point $x\in F(p)$ is a {\em flat point} if $F_x$ is a hexagon; it is
{\em positively curved point} if $F_x$ has two or four sides, and it is a {\em negatively curved point} if
$F_x$ has eight or more sides. Usually we will deal with just three kinds, the rectangular points, the flat or hexagonal points, and the
octagonal points.

We have the following important observation which shows that the two-manifold property is preserved by the
combination of pasting together and cutting out.

\begin{lemma}
In the situation of Theorem \ref{patchcut}, suppose $F$ is a two-manifold construction. Then $\widehat{G}$ is also a
two-manifold construction.
\end{lemma}

Suppose $F$ is a two-manifold construction with no twofold points. Suppose $x\in F(p)$ is a rectangular point (i.e.\ $F_x$ has
four edges), and suppose $\sigma$ is a full coherent scaffolding such that two of the edge germs at $x$ are in $\sigma ^o$.
Then there exists a copy of $W^{a,b}\subset F$ sending the origin to $x$, and by fullness and coherence of $\sigma$
we may assume that we
are in the situation of
Theorem \ref{patchcut}, meaning that the two edges of $V^{a,b}$ are straight and
the angle between them is open.
We may assume that it is a maximal such copy. The process of Theorem \ref{patchcut}
yields a new two-manifold construction $\widehat{G}$.

\subsection{The reduction process}
\label{sec-reduction}

We may now schematize the reduction process obtained by iterating the above operation.

We are going to define a sequence of two-manifold constructions
$F_i$ with full coherent scaffoldings $\sigma _i$. The first one $F_0$ will be obtained by trimming some
initial construction. We assume that there are never any twofold points.
We will be happy and the sequence will stop if we reach a construction $F_i$ which has no fourfold points,
so it is a pre-building.

Suppose we have constructed the $(F_i,\sigma _i)$ for $i\leq k$. We suppose that $F_k$ has no twofold points,
and that at every rectangular point, two of the edges are marked ``open''.
Choose a rectangular point $x\in F_k$ and let $W^{a,b}\subset F_k$ be a maximal copy with the origin corresponding to $x$,
such that the
edges of $V^{a,b}$ are straight. Applying the pasting together and cutting out process of Theorem \ref{patchcut},
we obtain a new two-manifold construction $F_{k+1}:= \widehat{G}$.

\begin{hypothesis}
\label{reductionhyp}
We assume $\widehat{G}$ has no twofold points,
and that there is a unique way to extend the scaffolding $\sigma_{\widehat{G}}$
to a full coherent scaffolding $\sigma _{k+1}$ on $F_{k+1}$ such that
there are two open edges at any new rectangular points.
\end{hypothesis}

\begin{corollary}
If this hypothesis is satisfied at each stage, then the
reduction process may be defined.
\end{corollary}

One of the main ideas which we would like to suggest is that this hypothesis will be true if there
are no BPS states.
How the reduction process works, and the relationship with BPS states,
will be discussed further in Section \ref{sec-progress} below.

\subsection{Recovering the pre-building}
\label{sec-recovering}

Suppose we start with a construction $F$ and scaffolding $\sigma$,
then trim it to obtain a two-manifold construction $F_0$ without twofold points.
Suppose then that Hypothesis \ref{reductionhyp} holds at each stage so we can define the reduction process,
and suppose that the process stops after a finite time with some $F_k$ which has no positively curved points.
Then we would like to say that this allows us to recover a pre-building with a map from $F$. We explain how this
works in the present subsection.

The first stage is to undo the initial trimming.
Put $G_0:=F$. Then apply Theorem \ref{th-trim} successively to obtain a sequence of constructions
$G_i$, with $G_{i-1}$ playing the role denoted $F$ in Theorem \ref{th-trim}.
Hence, $G_{i-1}$
is a pushout of $G_i$ with a standard parallelogram. If this process stops at a trimmed
construction $F_0=G_k$ then we may successively do the pushouts to get $G_{k-1}, G_{k-2},\ldots , G_0=F$.
Thus $F$ is obtained from $F_0$ by a sequence of pushouts with standard parallelograms.

Next, starting with $F_0$ we do a sequence of pasting and cutting operations to yield
two-manifold constructions $F_i$. Here, $F_i$ is obtained from $F_{i-1}$ by
applying Theorem \ref{patchcut}. In particular, we have a construction $\widetilde{F}_{i-1}$
which is on the one hand the result of a pasting operation applied to $F_{i-1}$, but on the
other hand is a pushout of $F_i$ along a standard parallelogram.

Define inductively a sequence of constructions $H_i$, starting with $H_0:= F$, as follows.
At each stage, we will have $F_i \subset H_i$, and $H_i$ is obtained from $F_i$ by a sequence
of pushouts along standard parallelograms
(to be precise this means pushouts along the inclusion $V^{a,b}\hookrightarrow P^{a,b}$).
This holds for $H_0$.

Suppose we know $H_{i-1}$. Then $\widetilde{F}_{i-1}$ is a pasting of $F_{i-1}$
as in Corollary \ref{patchcor}. Do the same pasting to $H_{i-1}$ instead of $F_{i-1}$, in other words
$$
H_i:= H_{i-1} \cup ^{F_{i-1}} \widetilde{F}_{i-1} .
$$
By the inductive hypothesis, $H_{i-1}$ is a pushout of $F_{i-1}$ by a series of
standard parallelograms, thus it follows that $H_i$ is a pushout of $\widetilde{F}_{i-1}$
by the same series. On the other hand,
$\widetilde{F}_{i-1}$ was a pushout of $F_i$ by a standard parallelogram.
Therefore we obtain the inductive hypothesis that $H_i$ is a pushout of $F_i$ along a series
of standard parallelograms.

We have a sequence of maps $H_{i-1}\rightarrow H_i$. These are quotient maps. Composing
them we obtain a series of maps $F\rightarrow H_i$. If $F$ was the initial construction for
$(X,\phi )$ then we would get in this way a $\phi$-harmonic map $X\rightarrow H_i$.

Under the hypothesis that the process leading to a sequence of $F_i$ stops at some finite $k$,
we obtain a construction $H_k$. If $F_k$ has no fourfold points,
we would like to say that $H_k$ is the pre-building; however,
in putting back in the pushouts with standard parallelograms, we might be introducing new fourfold points.
Therefore, as was done previously for the small
object argument, we should apply the folding-in process of Lemma \ref{foldingin}.
Precisely, having $F_k\subset H_k$ obtained by a sequence of pushouts by standard
parallelograms, apply the folding-in process
at each stage to transform the pushout into a good pushout as described in Lemma \ref{foldingin}
preserving the pre-building property. The resulting construction is a pre-building $\widetilde{H}$
accepting a map $F\rightarrow \widetilde{H}$. It is universal for maps to pre-buildings.

\begin{theorem}
Suppose that there are no BPS states and
Conjectures \ref{mainconj} and \ref{finitenessconj} hold. Then we obtain
a pre-building $B^{\rm pre}_{\phi}:=\widetilde{H}$ and a $\phi$-harmonic map
$$
h_{\phi}:X\rightarrow B^{\rm pre}_{\phi}
$$
which is universal for $\phi$-harmonic maps to pre-buildings containing extensions
to the enclosures $E_x$ used for the initial construction. The small object argument
of Theorem \ref{small-object} yields a versal $\phi$-harmonic map to a building.
\end{theorem}

The map to the building is only versal. Indeed, at
a stage of the construction of $B^{\rm pre}_{\phi}$
where we add a pushout with a standard parallelogram,
if we are given a map $V^{a,b}\rightarrow B$ which is not itself standard (say for example the
two edges coincide) then the extension to a map $P^{a,b}\rightarrow B$ is not unique.

\section{The BNR example, revisited}
\label{BNRrev}

In order to illustrate the procedure described above, we show how it works in the BNR example from \cite{KNPS}.
It was by considering this example in a new way that we came upon the above procedure.

Recall the picture from \cite[Figure 3]{KNPS} of the spectral network containing two singular points $b_1,b_2$.
At each singularity, there are two spectral network lines which meet lines from the opposite singularity,
and the resulting four segments delimit a diamond-shaped zone in the middle of the picture. There is a
single caustic $C$ joining $b_1$ to $b_2$ across the middle of this region.

Recall that we had found that a $\phi$-harmonic map to a building would fold together this middle region
along the caustic, identifying the two collision points. The resulting pre-building is conical with a single singularity
at the origin. Here, eight sectors form a two-manifold, with the collision being a negatively curved singular
point of total angle $480^{\circ}$. There are two additional sectors, forming a link across the middle of the octagon.
The image of the middle region goes into these sectors but is not surjective due to the folding along the caustic,
see  \cite[Figure 4]{KNPS}.

We now illustrate how this end picture comes about using the process described above. The first step is to create
the initial construction. For this, let us cover the caustic by some regions
as illustrated in Figure \ref{caustic-regions}.

\begin{figure}
\centering
\includegraphics[width=0.9\textwidth]{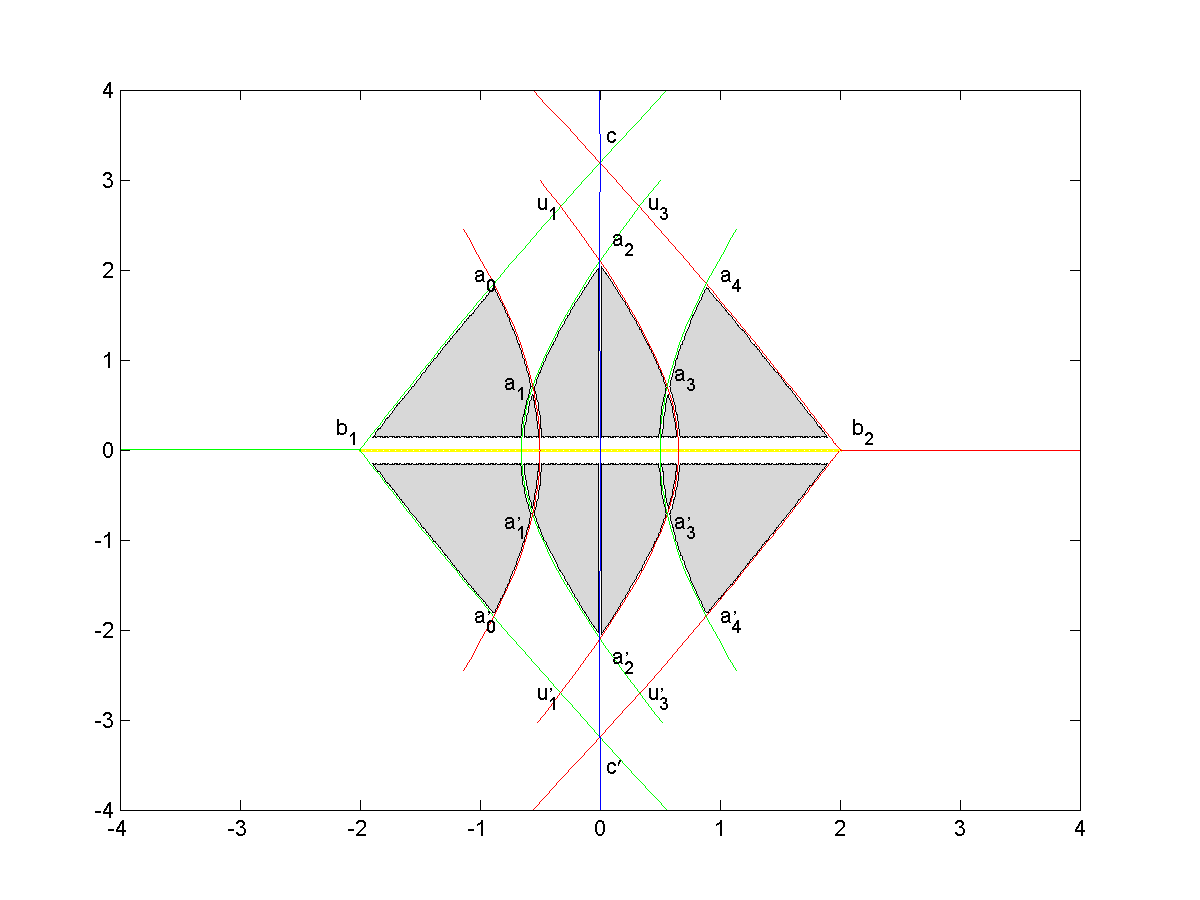}
\caption{Regions covering the caustic}
\label{caustic-regions}
\end{figure}

\medskip

The regions which cover the caustic are the ones which are bounded by the two foliation lines
going from $a_0$ to $a'_0$, resp. $a_2$ to $a'_2$, resp. $a_4$ to $a'_4$. The first and last ones
include the singularities $b_1$ and $b_2$.

Each of these regions is folded under any $\phi$-harmonic map. They are in fact $\Omega _{pq}$ regions (except
at the two endpoints). The image of the regions from Figure \ref{caustic-regions} under
the integration map to an apartment, are shown in Figure \ref{regions-image}.

The enclosures corresponding to these folded pieces are sketched into the picture completing the
shaded regions to full parallelograms.

The initial construction $Z$ consists of glueing together these enclosures covering the caustics, and then
adding enclosures around the remaining points which are not for the moment glued together. The
illustration of Figure \ref{regions-image} may be viewed as a picture of $Z$, where the
non-shaded
regions of the curve
correspond to two different sheets, an upper and a lower one corresponding to the upper and
lower regions in Figure \ref{caustic-regions}.
The completions of parallelograms just below the caustic are part of the initial construction $Z$
(these pieces should be considered
as shaded having only one sheet) but they are not images of points in $X$.
The zone below the dashed line is empty, that is to say it doesn't correspond to any points in $Z$,
in Figure \ref{regions-image}
and similarly for the subsequent ones.

\begin{figure}
\centering
\includegraphics[width=0.9\textwidth]{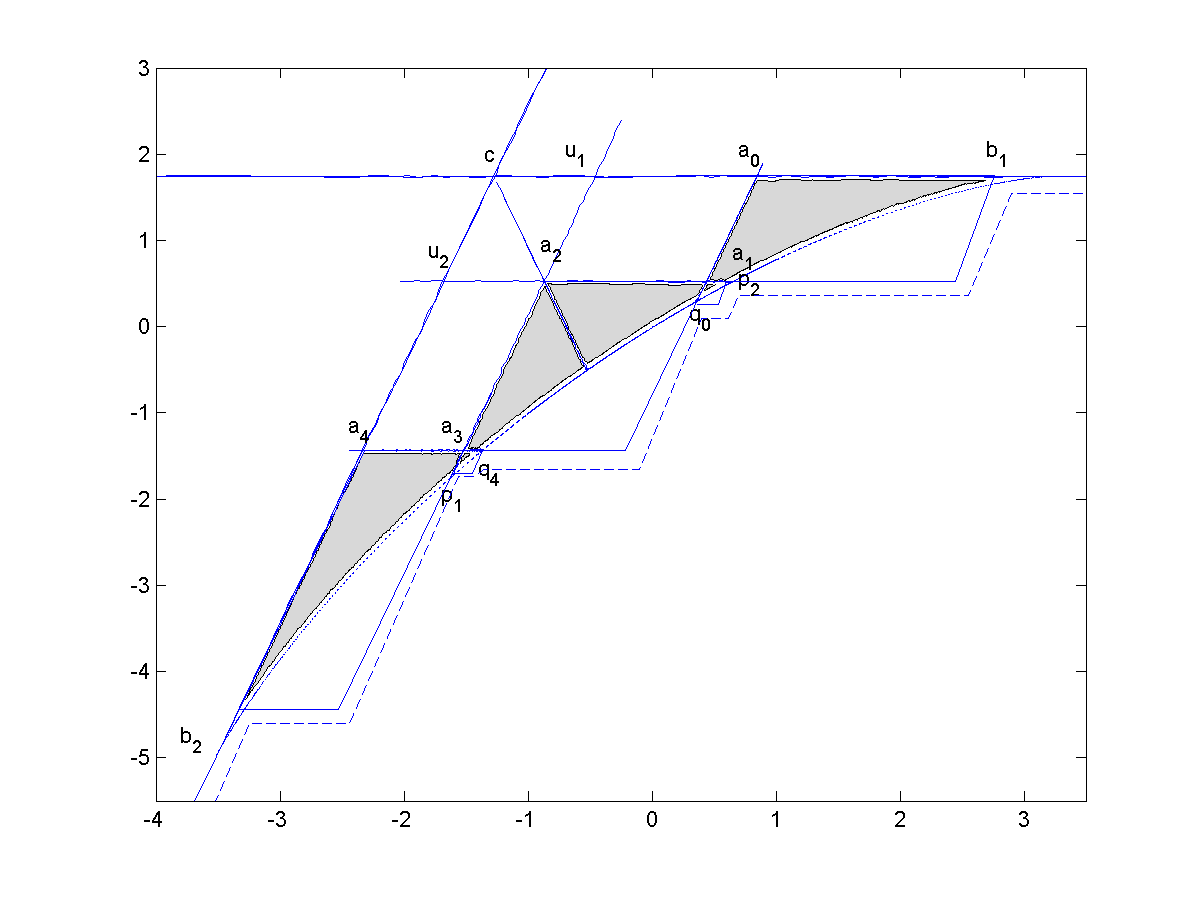}
\caption{Images of the regions}
\label{regions-image}
\end{figure}

\begin{figure}
\centering
\includegraphics[width=0.9\textwidth]{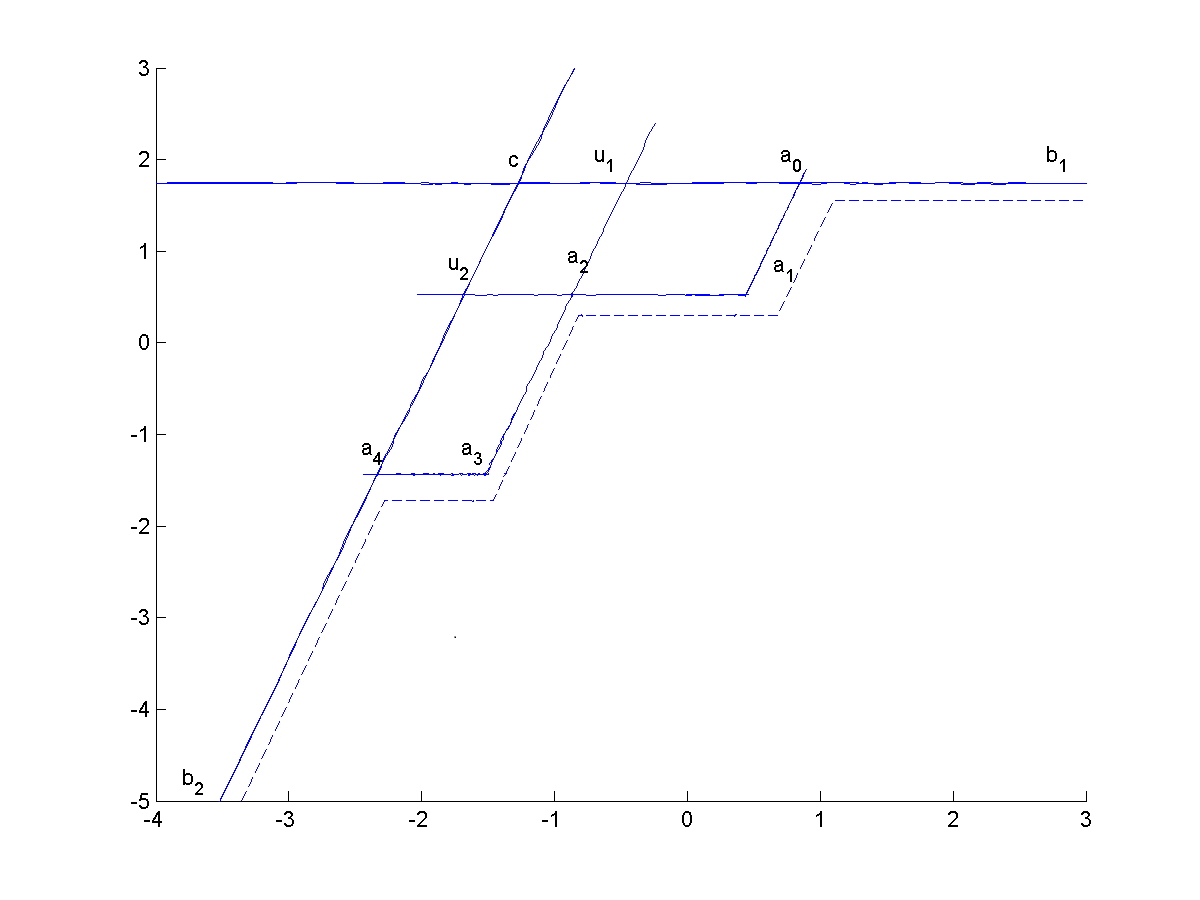}
\caption{After trimming}
\label{trimmed}
\end{figure}

Now $Z$ may be trimmed by taking out the regions containing folded pieces, leaving $Z_0$ as shown in
Figure \ref{trimmed}. Again there are two sheets, which are joined together along the
edge from $b_1$ to $b_2$ (and also along the rays pointing outward from $b_1$ and $b_2$).
In this case the edge where the two sheets are joined corresponds to the dotted line, and as
said above, the zone below the dotted line is empty.

In $Z_0$ we have five points $a_0,a_1,a_2,a_3,a_4$ which are not hexagonal.
The points $a_0,a_2,a_4$ are eightfold whereas $a_1,a_3$ are fourfold points. The segments $b_1a_0$ and
$a_4b_2$ are marked ``open'', due to the trivalent edges in the constructions which are placed at the singularities.
Indeed, on the outer side of $b_1$ one gets an open edge, because we are in the image of a neighborhood in $x$
as was illustrated in Figures \ref{littlenbds} and \ref{cornersegs} .
This open edge then propagates into the
segment $b_1a_0$ by Lemma \ref{straightcoh}, because in $Z_0$ all the points along this edge
are hexagonal, including $b_1$, and up to (but not including) $a_0$. The segment is straight because on either side
we are in the image of neighborhoods from $X$.

We may now show that the four segments $a_0a_1$, $a_1a_2$, $a_2a_3$ and $a_3a_4$ are marked ``fold''.
Consider for example the fourfold point $a_1$. Let $u_1$ and $u'_1$ be the
points which complete the two parallelograms spanned by $a_0, a_1,a_2$.
The four sectors around $a_1$ are $a_1a_0u_1$,
$a_1a_0u'_1$, $a_1a_2u_1$ and $a_1a_2u'_1$. Suppose we are given a map $Z_0\rightarrow B$ to a building,
coming from a $\phi$-harmonic map $h:X\rightarrow B$. The two sectors $a_1a_0u_1$ and $a_1a_2u_1$
come from two adjacent sectors at a point in $X$ (one of the lifts of the point $a_1$). Since the
map $h$ doesn't fold anything in $X-C$, we conclude that there is no folding along the edge $a_1u_1$.
Similarly for $a_1u'_1$. But since there are no cycles of length $\leq 4$ in the local spherical buildings of $B$,
the four sectors have to be collapsed somehow. Therefore, our map must fold along the segments $a_1a_0$
and $a_1a_2$.

The same argument at $a_3$ shows that the map must fold along $a_3a_2$ and $a_3a_4$.

Therefore the four segments $a_0a_1$, $a_1a_2$, $a_2a_3$ and $a_3a_4$ are marked ``fold''.

All other edges of $Z_0$ correspond to edges in $X$, so we conclude that all edges except for these four are marked ``open''.

We may now look at how the pasting-together process works. The first step is to paste together the
two parallelograms $a_0a_1a_2 u_1$ and $a_0a_1a_2 u'_1$. Similarly we paste together
$a_2a_3a_4 u_3$ and $a_2a_3a_4 u'_3$. The result is shown in Figure \ref{nextpaste}.

\begin{minipage}[h]{0.45\textwidth}
\centering
\includegraphics[width=0.9\textwidth]{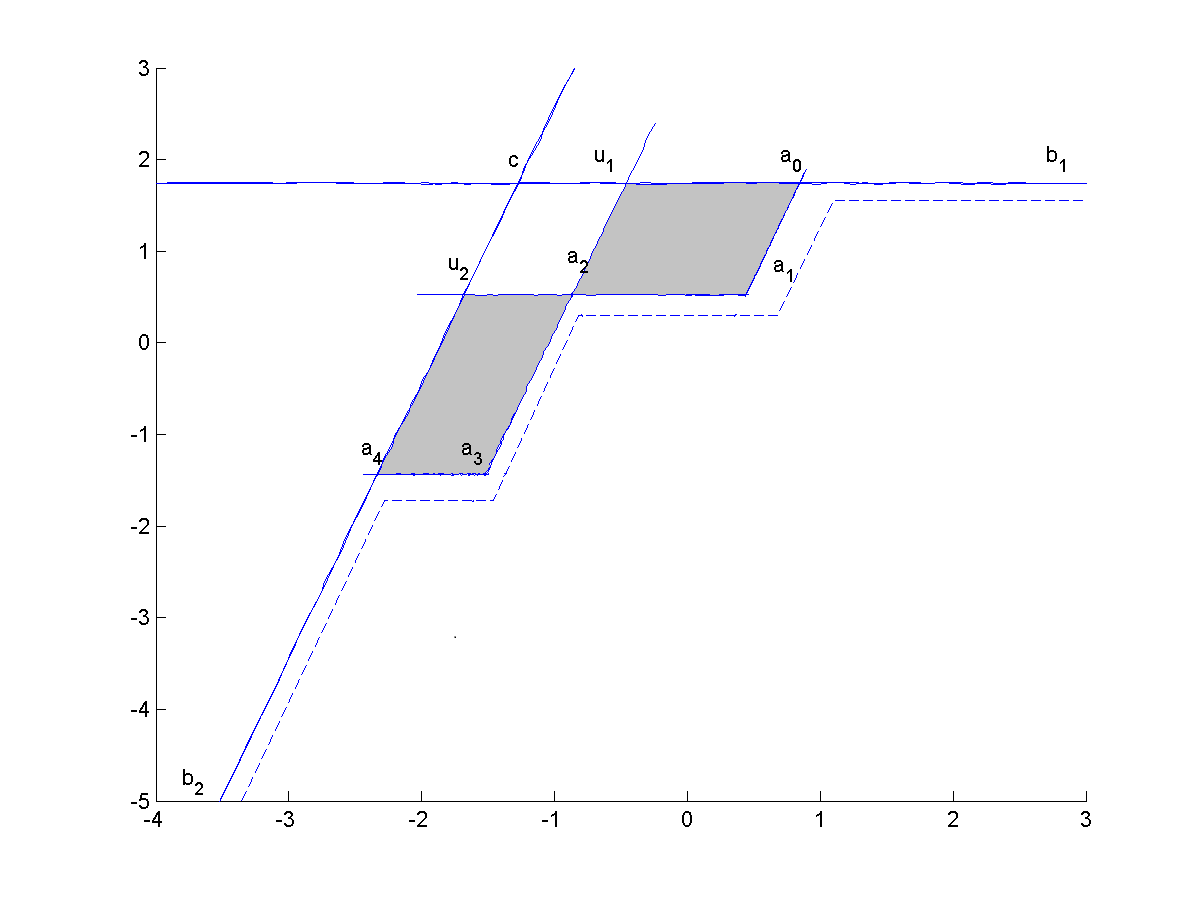}
\captionof{figure}{}
\label{nextpaste}
\end{minipage}
\begin{minipage}[h]{0.45\textwidth}
\centering
\includegraphics[width=0.9\textwidth]{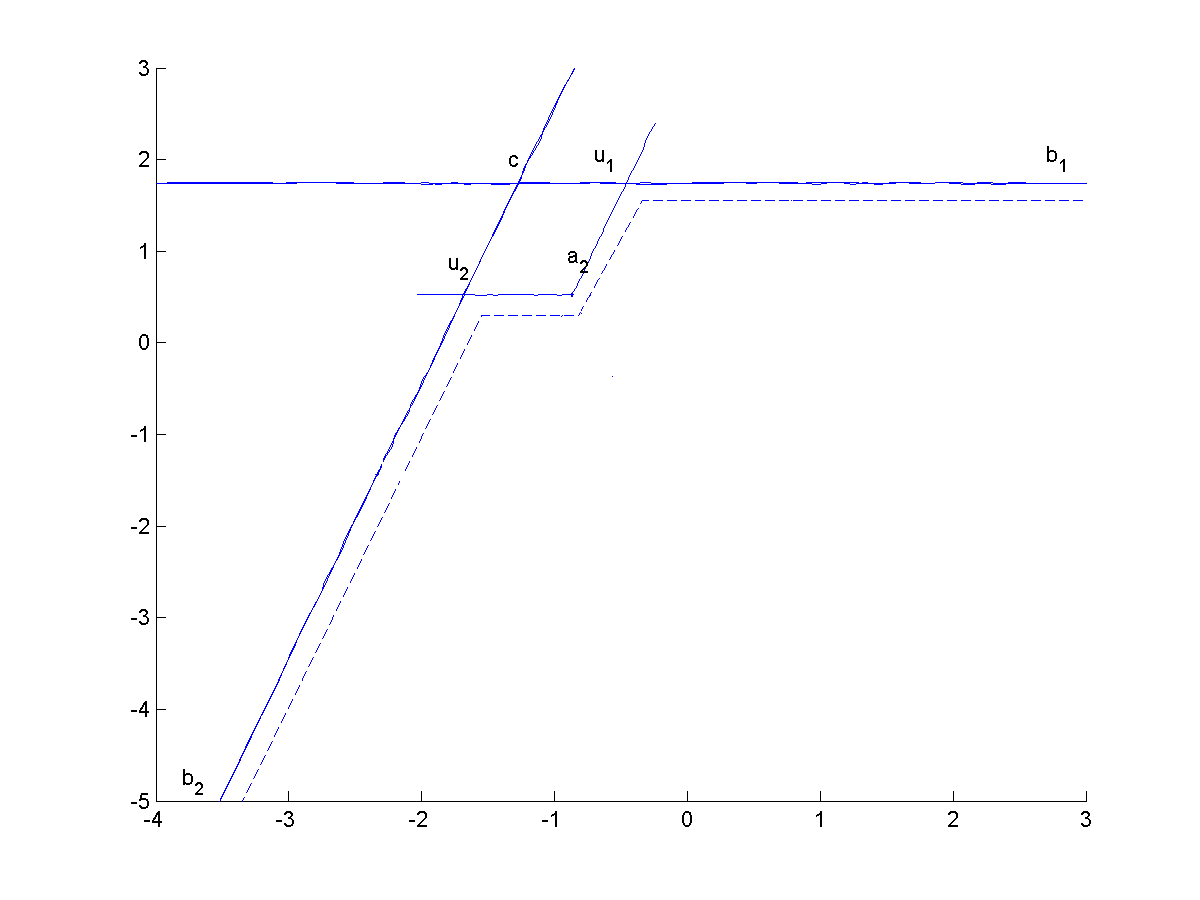}
\captionof{figure}{}
\label{nexttrim}
\end{minipage}

\medskip

Then we can cut out the parallelograms obtained from the above pasting. This gives the picture
shown in Figure \ref{nexttrim} which is
again a two-manifold construction $Z_1$. As before there are two sheets joined along the edge.

Now there is one remaining fourfold point at $a_2$ (notice that what was previously an eightfold point has
now become a fourfold point). Denote by $u_1$ and $u_3$ the images in $Z_1$ of the
pairs of points $(u_1,u'_1)$ and $(u_3,u'_3)$ respectively. By the same reasoning as before, the
segments $a_2u_1$ and $a_2u_3$ are marked ``fold'' whereas all the other edges are marked ``open''.

The next and last step of the pasting-together process is to paste together the two parallelograms
$u_1a_2u_3c$ and $u_1a_2u_3c'$ where $c$ and $c'$ are the two collision points from $X$. This gives
Figure \ref{lastpaste}.

\begin{minipage}[h]{0.45\textwidth}
\centering
\includegraphics[width=0.9\textwidth]{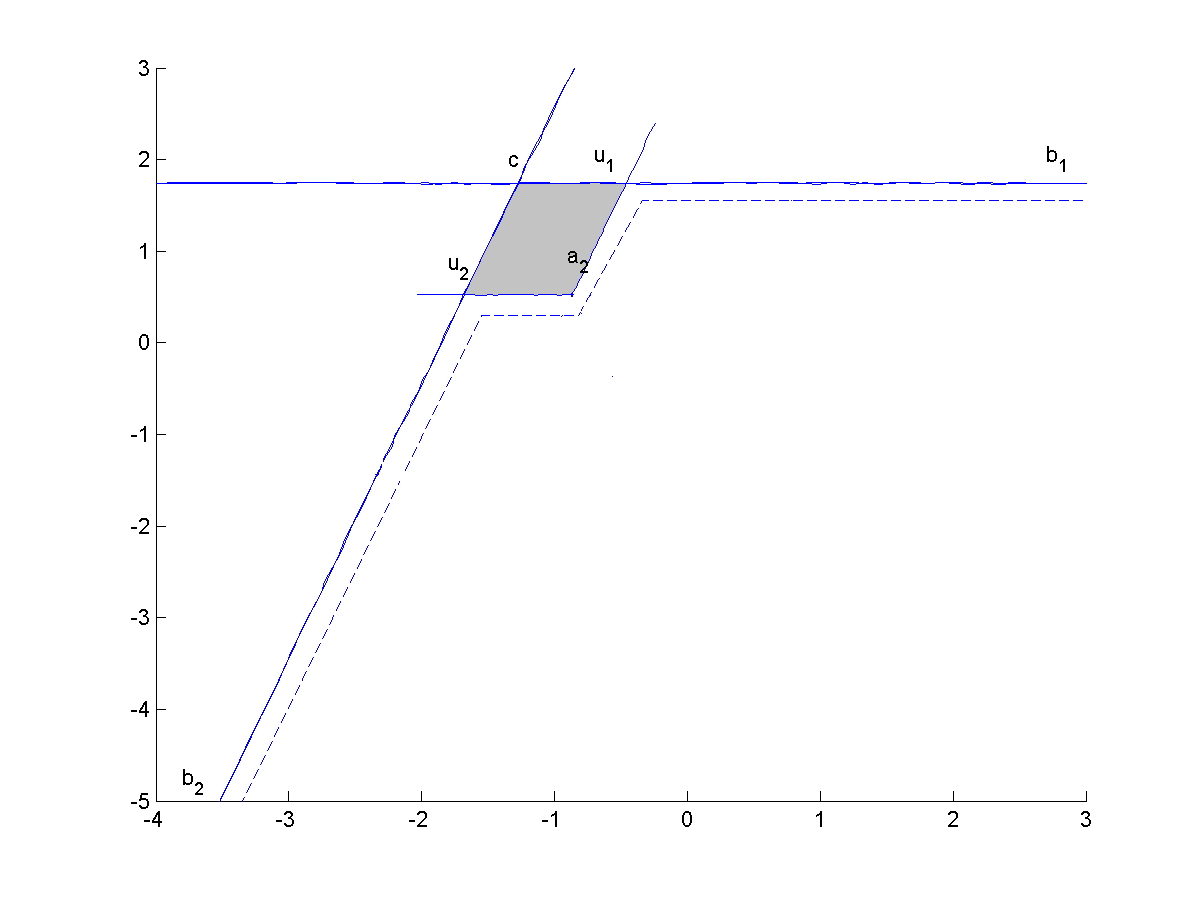}
\captionof{figure}{}
\label{lastpaste}
\end{minipage}
\begin{minipage}[h]{0.45\textwidth}
\centering
\includegraphics[width=0.9\textwidth]{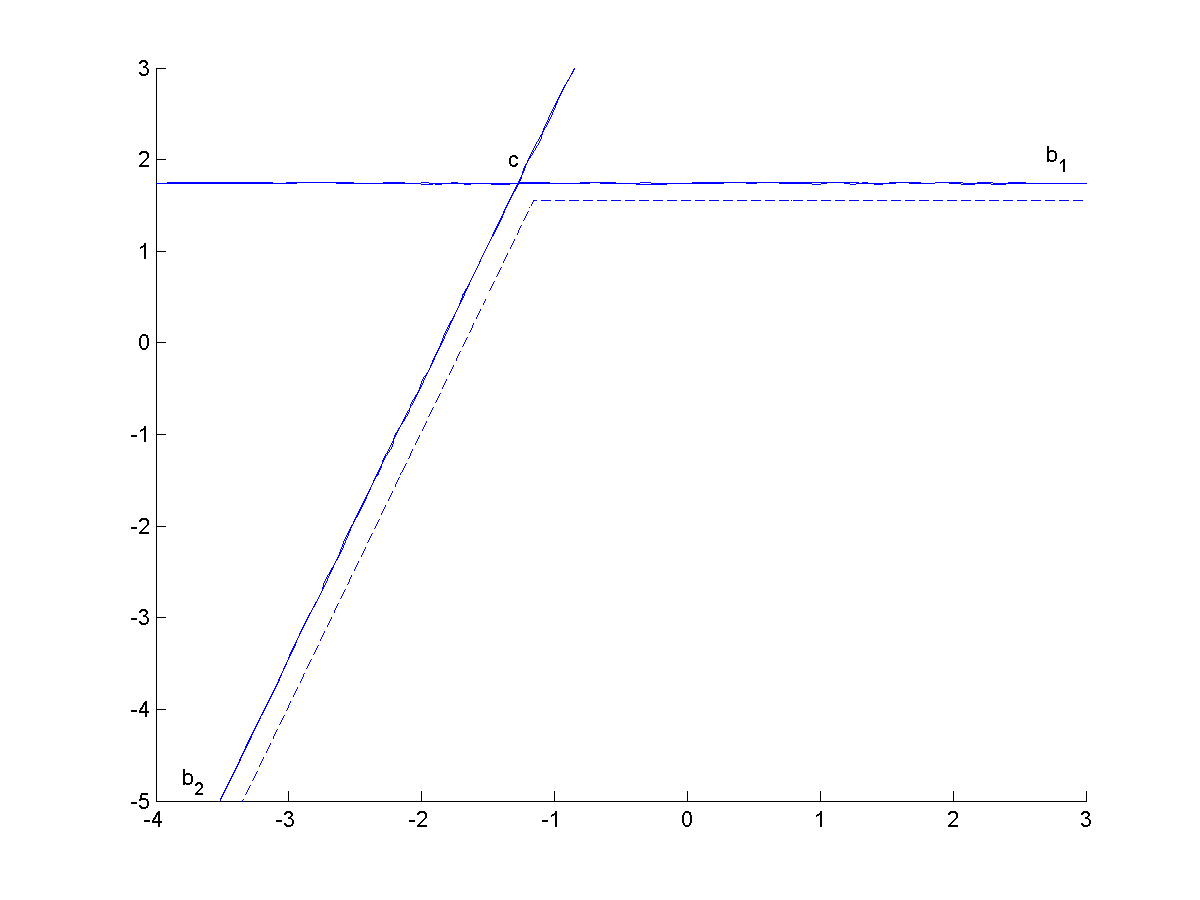}
\captionof{figure}{}
\label{lasttrim}
\end{minipage}

\medskip

Cutting out the parallelogram obtained from the pasting-together, gives the two-manifold construction $Z_2$ shown in Figure \ref{lasttrim}.

It has only a single eightfold point at the image of the two collision points which are now identified. There are no
fourfold points. Therefore we may back up and add back in all of the parallelograms which were removed by the
cutting-out processes.

This gives the picture shown in Figure \ref{bnrPB}
which is the universal pre-building accepting a map from $X$. Here as before, the unshaded regions
above and to the left correspond to two sheets; the shaded regions as well as the
remaining pieces of the parallelograms just next to the dotted line, are single sheets,
and everything below the dotted line is empty.

\begin{minipage}[h]{0.85\textwidth}
\centering
\includegraphics[width=0.9\textwidth]{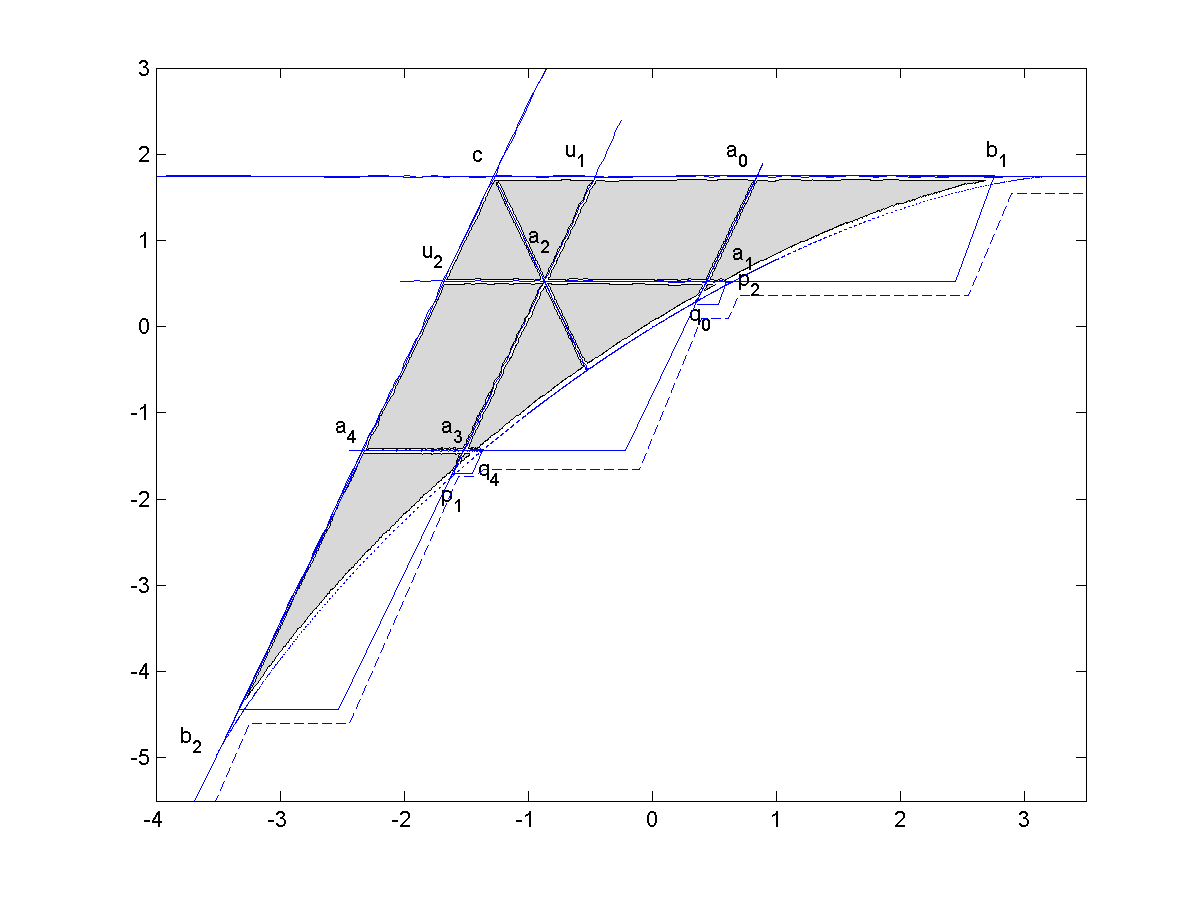}
\captionof{figure}{The pre-building}
\label{bnrPB}
\end{minipage}

\medskip

When we turn this into a building using the small object argument, we get back the universal building constructed in \cite{KNPS}.

This was a first illustration of how the reduction process introduced in Section \ref{sec-reduction}
leads to a universal pre-building.
In the next section we discuss some further aspects of the process in the general case.

\section{The process in general}
\label{sec-progress}

In general we have the following setup: starting with $(X,\phi )$ we first make the initial construction $Z$, then
trim it to get a two-manifold construction $Z_0$. Then we go through a sequence of steps of pasting-together
then cutting-out described in
Section \ref{sec-reduction}. This yields a sequence of two-manifold constructions $Z_i$
(these were denoted $F_i$ in Section \ref{sec-modifying} but we change to the notation $Z_i$ here in
order to think of them as modifications of the original surface $Z$). The construction $Z_0$
has a coherent full scaffolding by Lemma \ref{initialscaff}.  In order to proceed with the construction
at each step, we would like to know that after each operation the new $Z_i$ may still be given a uniquely
defined full coherent scaffolding, which is Hypothesis \ref{reductionhyp}.

If this hypothesis holds at each stage, then we can continue the operation and hope that it
converges locally at least. In this section we explain some ideas for how to understand more precisely the
sequence of constructions, and how to see Hypothesis \ref{reductionhyp} saying that
full scaffoldings will be determined at each step, if there are no BPS states.

\subsection{Properties of scaffoldings}
Let us first axiomatize some properties of our two-manifold constructions and their scaffoldings.

Our two-manifold constructions $Z_i$ are provided with scaffoldings in which
almost all of the edges are marked ``open'' or $o$. Those vertices in the local spherical buildings which are marked ``fold'' or $f$, are parts of straight edges (that is, edges such that at each point, at least one side of the edge has both
directions transverse to that edge, marked $o$). The marking is the same along the straight edge.

\noindent
{\bf Fourfold points}---
We require that, at any fourfold point there should be two opposite edges marked $o$, and two opposite edges
marked $f$:

$$
\setlength{\unitlength}{.3mm}
\begin{picture}(200,120)

\put(100,60){\circle*{2}}

\put(100,60){\line(1,0){50}}
\put(100,60){\line(-1,0){50}}

\linethickness{.5mm}
\put(100,60){\line(0,1){50}}
\put(100,60){\line(0,-1){50}}
\thinlines

\put(95,112){$f$}
\put(97,0){$f$}
\put(157,57){$o$}
\put(40,57){$o$}

\end{picture}
$$

The configuration with all four edges folded is admissible from the point of
view of coherence, but throughout our procedure we conjecture that it should not occur.

\noindent
{\bf Hexagonal points}---
At a hexagonal point, the possibilities are as follows. First, either all edges are open
or two opposite edges are folded.

\begin{minipage}[h]{0.45\textwidth}
$$
\setlength{\unitlength}{.3mm}
\begin{picture}(200,120)

\put(100,60){\circle*{2}}

\put(100,60){\line(1,0){50}}
\put (157,57){$o$}

\put(100,60){\line(-1,0){50}}

\put(40,57){$o$}

\qbezier(100,60)(113,81)(126,102)

\put(128,104){$o$}

\qbezier(100,60)(87,81)(74,102)
\put(64,104){$o$}

\qbezier(100,60)(113,39)(126,18)

\put(128,10){$o$}

\qbezier(100,60)(87,39)(74,18)
\put(64,10){$o$}

\end{picture}
$$
\end{minipage}
\begin{minipage}[h]{0.45\textwidth}
$$
\setlength{\unitlength}{.3mm}
\begin{picture}(200,120)

\put(100,60){\circle*{2}}

\linethickness{.5mm}
\put(100,60){\line(1,0){50}}
\thinlines
\put (157,57){$f$}

\linethickness{.5mm}
\put(100,60){\line(-1,0){50}}
\thinlines
\put(40,57){$f$}

\qbezier(100,60)(113,81)(126,102)

\put(128,104){$o$}

\qbezier(100,60)(87,81)(74,102)
\put(64,104){$o$}

\qbezier(100,60)(113,39)(126,18)

\put(128,10){$o$}

\qbezier(100,60)(87,39)(74,18)
\put(64,10){$o$}

\end{picture}
$$
\end{minipage}

\medskip

\noindent
It is also possible to have three $o$ and three $f$
alternating; this configuration should actually be considered as a superposition of a $4v$ point
plus an $8v$ point starting a new postcaustic curve. Furthermore, one
additional edge can also be folded.

\begin{minipage}[h]{0.45\textwidth}
$$
\setlength{\unitlength}{.3mm}
\begin{picture}(200,120)

\put(100,60){\circle*{2}}

\put(100,60){\line(1,0){50}}
\put (157,57){$o$}

\linethickness{.5mm}
\put(100,60){\line(-1,0){50}}
\thinlines
\put(40,57){$f$}

\linethickness{.5mm}
\qbezier(100,60)(113,81)(126,102)
\thinlines
\put(128,104){$f$}

\qbezier(100,60)(87,81)(74,102)
\put(64,104){$o$}

\linethickness{.5mm}
\qbezier(100,60)(113,39)(126,18)
\thinlines
\put(128,10){$f$}

\qbezier(100,60)(87,39)(74,18)
\put(64,10){$o$}

\end{picture}
$$
\end{minipage}
\begin{minipage}[h]{0.45\textwidth}
$$
\setlength{\unitlength}{.3mm}
\begin{picture}(200,120)

\put(100,60){\circle*{2}}

\linethickness{.5mm}
\put(100,60){\line(1,0){50}}
\thinlines
\put (157,57){$f$}

\linethickness{.5mm}
\put(100,60){\line(-1,0){50}}
\thinlines
\put(40,57){$f$}

\linethickness{.5mm}
\qbezier(100,60)(113,81)(126,102)
\thinlines
\put(128,104){$f$}

\qbezier(100,60)(87,81)(74,102)
\put(64,104){$o$}

\linethickness{.5mm}
\qbezier(100,60)(113,39)(126,18)
\thinlines
\put(128,10){$f$}

\qbezier(100,60)(87,39)(74,18)
\put(64,10){$o$}

\end{picture}
$$
\end{minipage}

\medskip

Finally, a hexagonal point can have all edges folded. Again, we conjecture that this doesn't appear
in our procedure.

\noindent
{\bf Eightfold points}---
At an eightfold point, there are the following main possibilities. First, a BNR point
with all edges open. Then, a point with a single fold line,
as happens at the end of a postcaustic curve, initially coming from a
singularity.

\begin{minipage}[h]{0.45\textwidth}
$$
\setlength{\unitlength}{.3mm}
\begin{picture}(200,120)

\put(100,60){\circle*{2}}

\put(100,60){\line(0,1){50}}
\put(97,112){$o$}

\qbezier(100,60)(120,80)(140,100)
\put(143,99){$o$}

\put(100,60){\line(1,0){50}}
\put(157,57){$o$}

\qbezier(100,60)(120,40)(140,20)
\put(143,14){$o$}

\put(100,60){\line(0,-1){50}}
\put(97,0){$o$}

\qbezier(100,60)(80,40)(60,20)
\put(52,14){$o$}

\put(100,60){\line(-1,0){50}}
\put(40,57){$o$}

\qbezier(100,60)(80,80)(60,100)
\put(52,99){$o$}

\end{picture}
$$
\end{minipage}
\begin{minipage}[h]{0.45\textwidth}
$$
\setlength{\unitlength}{.3mm}
\begin{picture}(200,120)

\put(100,60){\circle*{2}}

\linethickness{.5mm}
\put(100,60){\line(0,1){50}}
\thinlines
\put(95,112){$f$}

\qbezier(100,60)(120,80)(140,100)
\put(143,99){$o$}

\put(100,60){\line(1,0){50}}
\put(157,57){$o$}

\qbezier(100,60)(120,40)(140,20)
\put(143,14){$o$}

\put(100,60){\line(0,-1){50}}
\put(97,0){$o$}

\qbezier(100,60)(80,40)(60,20)
\put(52,14){$o$}

\put(100,60){\line(-1,0){50}}
\put(40,57){$o$}

\qbezier(100,60)(80,80)(60,100)
\put(52,99){$o$}

\end{picture}
$$
\end{minipage}

\medskip

\noindent
There are two ways of having a post-caustic pass through the eightfold point, either
two opposite fold lines, as is most standard; or
two adjacent fold lines,.

\begin{minipage}[h]{0.45\textwidth}
$$
\setlength{\unitlength}{.3mm}
\begin{picture}(200,120)

\put(100,60){\circle*{2}}

\linethickness{.5mm}
\put(100,60){\line(0,1){50}}
\thinlines
\put(95,112){$f$}

\qbezier(100,60)(120,80)(140,100)
\put(143,99){$o$}

\put(100,60){\line(1,0){50}}
\put(157,57){$o$}

\qbezier(100,60)(120,40)(140,20)
\put(143,14){$o$}

\linethickness{.5mm}
\put(100,60){\line(0,-1){50}}
\thinlines
\put(97,0){$f$}

\qbezier(100,60)(80,40)(60,20)
\put(52,14){$o$}

\put(100,60){\line(-1,0){50}}
\put(40,57){$o$}

\qbezier(100,60)(80,80)(60,100)
\put(52,99){$o$}

\end{picture}
$$
\end{minipage}
\begin{minipage}[h]{0.45\textwidth}
$$
\setlength{\unitlength}{.3mm}
\begin{picture}(200,120)

\put(100,60){\circle*{2}}

\linethickness{.5mm}
\put(100,60){\line(0,1){50}}
\thinlines
\put(95,112){$f$}

\linethickness{.5mm}
\qbezier(100,60)(120,80)(140,100)
\thinlines
\put(143,99){$f$}

\put(100,60){\line(1,0){50}}
\put(157,57){$o$}

\qbezier(100,60)(120,40)(140,20)
\put(143,14){$o$}

\put(100,60){\line(0,-1){50}}
\put(97,0){$o$}

\qbezier(100,60)(80,40)(60,20)
\put(52,14){$o$}

\put(100,60){\line(-1,0){50}}
\put(40,57){$o$}

\qbezier(100,60)(80,80)(60,100)
\put(52,99){$o$}

\end{picture}
$$
\end{minipage}

\medskip

\noindent
Notice that when the two adjacent fold lines get folded up, the image is a hexagon, with one sector covered three times by passing over and back.

This picture with two adjacent fold lines comes up, in particular,
in the initial construction when we have a crossing of two caustic lines. It can also
arise sometime later in the procedure.

\noindent
{\em Other eightfold points}---several other configurations are also admissible from the point of view of
coherence, although we feel that they probably will not occur in the procedure. It is left to the reader to
enumerate these.

This finishes the description of the expected properties of open and fold markings of our scaffolding.
The set of possible local pictures listed above covers the situations which we meet along the various
steps of the process. A scaffolding of a two-manifold construction
will be called {\em admissible} if its local pictures are in the above list (completed as per the preceding paragraph).

\begin{proposition}
If $Z_i$ is a two-manifold construction with a full coherent scaffolding, then
the scaffolding is admissible.
\end{proposition}

\begin{corollary}
In the two-manifold construction obtained by trimming the initial construction, and then
in the two-manifold constructions which are met along the way of our process (assuming that
there are no BPS states and Hypothesis \ref{reductionhyp} holds), the full coherent
scaffolding is always admissible.
\end{corollary}

\subsection{Post-caustics}

It is a consequence of the above list of local pictures allowed for an admissible scaffolding, that the collection of points which
have an edge marked ``fold'' in the scaffolding is organized into a collection of piecewise linear
curves marked as folded edges. Furthermore, it follows that the
endpoints of the curves are $8v$ points (or $6v$ triple points), and along a curve the successive segments
have endpoints which alternate as $8v, 4v, 8v, \ldots , 4v, 8v, 4v, 8v$. All points outside of these
are hexagonal i.e. flat.

Call these curves ``post-caustic'' curves, because they arise in the initial construction as approximations
to the caustics.

We may think of $Z_i$ as a locally flat surface but with singularities: positive curvature at the $4v$ points (total angle $240^{\circ}$)
and negative
curvature (total angle $480^{\circ}$) at the $8v$ points. The total curvature of a post-caustic curve is equivalent to one
$8v$ vertex.

\subsection{Determining new markings}

The next step is to start trying to remove the fourfold points by using the pasting-together
construction of Section \ref{sec-pasting}. Suppose we have gotten to a two-manifold construction $Z_i$.

At a fourfold point we choose a
standard pair of parallelograms to be glued together.
It may be seen from the $\Omega _{pq}$ argument that the two parallelograms must go to the same place in
any harmonic map to a building.
Here, by a standard parallelogram we mean one in which all of the interior edges, including interior edges
along the sides, are
labeled as open. Thus in a map to a building the paths which look noncritical, must actually map to noncritical paths. This applies in particular to the two outer edges which are common to both parallelograms.
They form a noncritical path joining the two endpoints of the parallelograms $p$ and $q$, and
the $\Omega _{pq}$ argument works as usual. Both parallelograms are swept out by noncritical paths so
the two images in any map to a building must be identified.

Let $Y_{i+1}$ denote the result of pasting together the two parallelograms.  The next two-manifold
construction $Z_{i+1}$ is then obtained by cutting out the resulting parallelogram from $Y_{i+1}$.

Denote by the ``common edges'' of the parallelograms the two edges which are common and which start from
the fourfold point; the ``glued edges'' are the other two edges of each parallelogram.

We assume that at least one of the two edge lengths of the parallelogram is maximal. It means that
a non-hexagonal point, or possibly a hexagonal point with nontrivial marking, was encountered along
some glued edge of at least one of the parallelograms.

In the new two-manifold construction $Z_{i+1}$, the only edges which don't benefit from the
previous scaffolding are the two glued edges. Therefore, the main problem is how to determine the
new marking along these edges. We illustrate this in the next section.

\subsection{A sample step}
\label{sec-example}

Let us consider an example of pasting with a singularity on the edge.
We draw this in Figure \ref{pasting}.
The two parallelograms to be pasted together are $PAQB$ and $PAQB'$.
They share two common edges, $PA$ and $QA$. These edges come together at a fourfold
point $A$, and the other two edges at $A$ are marked ``open'', so the edges $PA$ and $QA$
are marked ``fold''; hence they have been drawn as thick lines.  The interior edges of
both parallelograms are all marked ``open'', in particular the four sides of
each one are straight. Thus, the two parallelograms
must be mapped to the same place under any harmonic map to a building, and we will paste them together
and cut them out. Pasting together identifies the points $B$ and $B'$, we denote the image just by $B$ too.
After cutting out the parallelogram we get a new two-manifold shown in Figure \ref{newtwo}

\begin{figure}
\begin{minipage}[h]{0.45\textwidth}
$$
\setlength{\unitlength}{.3mm}
\begin{picture}(200,120)

\put(40,110){\circle*{2}}
\put(120,110){\circle*{2}}
\put(94,29){\circle*{2}}
\put(174,29){\circle*{2}}
\put(70,65){\circle*{2}}

\put(101,34){\circle*{2}}

\linethickness{.5mm}
\qbezier(20,110)(20,110)(120,110)
\qbezier(120,110)(120,110)(180,20)
\thinlines

\qbezier(40,110)(40,110)(94,29)
\qbezier(94,29)(94,29)(174,29)

\qbezier(40,110)(40,110)(101,34)
\qbezier(101,34)(174,29)(174,29)

\qbezier(70,65)(70,65)(80,65)
\qbezier(70,65)(70,65)(76,74)

\qbezier(70,65)(70,65)(70,56)
\qbezier(70,65)(70,65)(65,58)
\qbezier(70,65)(70,65)(61,63)
\qbezier(70,65)(70,65)(63,70)

\put(38,112){\ensuremath{\scriptstyle P}}
\put(120,112){\ensuremath{\scriptstyle A}}
\put(175,31){\ensuremath{\scriptstyle Q}}
\put(50,58){\ensuremath{\scriptstyle R}}

\put(85,26){\ensuremath{\scriptstyle B}}

\put(104,35){\ensuremath{\scriptstyle B'}}

\qbezier(20,110)(45,-5)(180,20)

\qbezier(20,110)(35,-20)(180,20)

\end{picture}
$$
\captionof{figure}{}
\label{pasting}
\end{minipage}
\begin{minipage}[h]{0.45\textwidth}
$$
\setlength{\unitlength}{.3mm}
\begin{picture}(200,120)

\put(40,110){\circle*{2}}

\put(94,29){\circle*{2}}
\put(174,29){\circle*{2}}
\put(70,65){\circle*{2}}

\linethickness{.5mm}
\qbezier(20,110)(20,110)(40,110)
\qbezier(174,29)(180,20)(180,20)

\qbezier(40,110)(40,110)(70,65)
\qbezier(94,29)(94,29)(174,29)

\thinlines

\qbezier(70,65)(70,65)(94,29)

\qbezier(70,65)(70,65)(70,56)
\qbezier(70,65)(70,65)(65,58)
\qbezier(70,65)(70,65)(61,63)
\qbezier(70,65)(70,65)(63,70)

\put(38,112){\ensuremath{\scriptstyle P}}

\put(175,31){\ensuremath{\scriptstyle Q}}
\put(50,58){\ensuremath{\scriptstyle R}}

\put(85,26){\ensuremath{\scriptstyle B}}

\qbezier(20,110)(40,-10)(180,20)

\qbezier(20,110)(35,-15)(180,20)

\end{picture}
$$
\captionof{figure}{}
\label{newtwo}
\end{minipage}
\end{figure}

The eightfold point $R$ lies on the edge $PB$ of the first parallelogram, this is constraining the
width of the parallelogram to be maximal. Let $R'$ denote the corresponding point on the other
parallelogram; we are assuming that $R'$ is a regular hexagonal point. Now $R$ and $R'$ are also
identified in the new two-manifold and we let their image be called $R$ too.

In the picture we have drawn, we are assuming that $P$ and $Q$ are hexagonal points and
the folded edges $PA$ and $QA$ extend beyond the
points $P$ and $Q$ respectively. In the new two-manifold, $P$ and $Q$ are fourfold points,
from which it follows that the edges $QB$ and $PR$ are
folded.

\noindent
{\bf Main question:} what is the marking of the edge $RB$?

Let us look in the spherical construction at the point $R$. At $R'$ we have a hexagon with all
directions marked ``open'', whereas at the original point
$R$ we are assuming there is an eightfold point. The
two directions which are interior to the parallelogram are marked ``open'' by the hypothesis
that the parallelogram is standard. We may picture the
original point $R$ as follows, with the edges going towards $P$ and $B$ marked $p$ and $b$ respectively.

$$
\setlength{\unitlength}{.3mm}
\begin{picture}(200,120)

\put(100,60){\circle*{2}}

%\linethickness{.5mm}
\put(100,60){\line(0,1){50}}
%\thinlines
\put(94,115){$p(o)$}

\qbezier(100,60)(120,80)(140,100)
\put(143,99){$o$}

\put(100,60){\line(1,0){50}}
\put(157,57){$o$}

\qbezier(100,60)(120,40)(140,20)
\put(143,14){$b$}

\put(100,60){\line(0,-1){50}}
\put(96,0){$s_4$}

\qbezier(100,60)(80,40)(60,20)
\put(51,14){$s_3$}

\put(100,60){\line(-1,0){50}}
\put(37,57){$s_2$}

\qbezier(100,60)(80,80)(60,100)
\put(48,102){$s_1$}

\end{picture}
$$

Notice that the edge $p$ should be marked ``open'' otherwise we would have a non-admissible configuration
at $P$.
There are several different possibilities for the labels $s_i$ as well as that of $b$,
as will be discussed shortly.

At the new point $R$, the three interior sectors have been replaced by the three sectors which
were exterior to the parallelogram at the point $R'$. The two directions inside here are again
marked ``open'' because of our hypothesis that $R'$ was a flat hexagonal point.
We are left with a diagram of the following form:

$$
\setlength{\unitlength}{.3mm}
\begin{picture}(200,120)

\put(100,60){\circle*{2}}

\linethickness{.5mm}
\put(100,60){\line(0,1){50}}
\thinlines
\put(94,115){$p(f)$}

\qbezier(100,60)(120,80)(140,100)
\put(143,99){$o$}

\put(100,60){\line(1,0){50}}
\put(157,57){$o$}

\qbezier(100,60)(120,40)(140,20)
\put(143,14){$b(?)$}

\put(100,60){\line(0,-1){50}}
\put(96,0){$s_4$}

\qbezier(100,60)(80,40)(60,20)
\put(51,14){$s_3$}

\put(100,60){\line(-1,0){50}}
\put(37,57){$s_2$}

\qbezier(100,60)(80,80)(60,100)
\put(48,102){$s_1$}

\end{picture}
$$
The direction $p$ is now marked ``fold'' as discussed previously.

\begin{lemma}
\label{sampledet}
In the above situation, the markings of $s_2$, $s_3$ and $s_4$ determine the marking of $b$.
\end{lemma}
\begin{proof}
Since $p$ is marked ``fold'', it has to be folded leading to a hexagonal point. The edge germs
to the left and right of $p$ are joined together into a single one marked $s$ below, and we don't know
how that one will be further folded. The hexagon may be pictured as:
$$
\setlength{\unitlength}{.3mm}
\begin{picture}(200,120)

\put(100,60){\circle*{2}}

\put(100,60){\line(1,0){50}}
\put (157,57){$o$}

\put(100,60){\line(-1,0){50}}
\put(40,57){$s_3$}

\qbezier(100,60)(113,81)(126,102)

\put(128,104){$s$}

\qbezier(100,60)(87,81)(74,102)
\put(64,104){$s_2$}

\qbezier(100,60)(113,39)(126,18)

\put(128,10){$b$}

\qbezier(100,60)(87,39)(74,18)
\put(64,10){$s_4$}

\end{picture}
$$
If $s_2$, $s_3$ and $s_4$ are open, then it follows that the full hexagon must be open.
If either $s_2$ or $s_4$ is folded, then from the given open edge we cannot have a triple fold,
so the opposite edge must also be folded. If $s_2$ is folded then
it follows that $b$ must be folded. If $s_4$ is folded then $s$ must be folded. The hexagon then
folds to a germ of half-apartment; and one can see $b$ would be folded if and only if
$s_3$ is folded.
On the other hand, if $s_3$ is folded, then since the
edge opposite to it is open, it follows that $s$ and $b$ must be folded. This completes the
determination of the marking of $b$.
\end{proof}

Of course this lemma treats just one of the various possible situations which can arise.
We expect that similar considerations should hold in the other cases.

In our example, since the
edge $RB$ is straight, the marking of $b$ propagates along the full edge and we conclude that
the marking of $RB$ is determined; for our example,
that completes the determination of the full scaffolding on the new two-manifold.

By coherence one sees that there were only a few possibilities for the folded edges at $R$,
and the reader may find it interesting to study what happens to the post-caustic curves in these
cases. Suppose for example that the marking of $b$ determined by the lemma is ``open''.
The possibilities for $R$ are: a BNR point with no folding;
a single folded edge $s_1$;
two opposite folded edges $s_1$ and $b$; or
two adjacent folded edges $s_4$ and $b$. In the latter two cases, the edge $BR$ was folded in
the original construction, and this fold must continue past $B$. In the new construction,
the edge $BR$ is no longer folded but the part of the post-caustic after $B$ will still be present.
Therefore, at the point $B$ in
the new construction we obtain two adjacent folded edges, this continuation together
with the new fold along $BQ$. Also in these two cases, at the new point $R$ we have two
folded edges, $p$ and either $s_1$ (adjacent) or $s_4$ (opposite). The new post-caustic coming from the point $P$ thus attaches onto the piece of post-caustic that used to come from one side of $B$, whereas the old
post-caustic on the other side of $B$ is now attached to the piece going through $Q$.

\subsection{The indeterminate case}
\label{sec-indeterminate}

In a less generic situation, the new marking might not be well determined. Consider a picture similar
to the previous one, but where the post-caustic bounding the parallelogram ends at an eightfold
point $P$
as shown in Figure \ref{pasting2}.

\begin{figure}
\begin{minipage}[h]{0.45\textwidth}
$$
\setlength{\unitlength}{.3mm}
\begin{picture}(200,130)

\put(40,110){\circle*{2}}
\put(120,110){\circle*{2}}
\put(94,29){\circle*{2}}
\put(174,29){\circle*{2}}
\put(70,65){\circle*{2}}

\put(101,34){\circle*{2}}

\linethickness{.5mm}
\qbezier(40,110)(40,110)(120,110)
\qbezier(120,110)(120,110)(180,20)
\thinlines

\qbezier(40,110)(40,110)(30,128)

\qbezier(40,110)(40,110)(94,29)
\qbezier(94,29)(94,29)(174,29)

\qbezier(40,110)(40,110)(101,34)
\qbezier(101,34)(174,29)(174,29)

\qbezier(70,65)(70,65)(80,65)
\qbezier(70,65)(70,65)(76,74)

\qbezier(70,65)(70,65)(70,56)
\qbezier(70,65)(70,65)(65,58)
\qbezier(70,65)(70,65)(61,63)
\qbezier(70,65)(70,65)(63,70)

\put(40,112){\ensuremath{\scriptstyle P}}
\put(120,112){\ensuremath{\scriptstyle A}}
\put(175,31){\ensuremath{\scriptstyle Q}}
\put(50,58){\ensuremath{\scriptstyle R}}

\put(85,26){\ensuremath{\scriptstyle B}}

\put(104,35){\ensuremath{\scriptstyle B'}}

\qbezier(20,100)(45,-5)(180,20)

\qbezier(12,97)(35,-20)(180,20)

\qbezier(30,128)(20,116)(20,100)
\qbezier(30,128)(12,115)(12,97)

\end{picture}
$$
\captionof{figure}{}
\label{pasting2}
\end{minipage}
\begin{minipage}[h]{0.45\textwidth}
$$
\setlength{\unitlength}{.3mm}
\begin{picture}(200,130)

\put(40,110){\circle*{2}}

\put(94,29){\circle*{2}}
\put(174,29){\circle*{2}}
\put(70,65){\circle*{2}}

\qbezier(40,110)(40,110)(70,65)

\qbezier(40,110)(40,110)(30,128)

\linethickness{.5mm}

\qbezier(174,29)(180,20)(180,20)

\qbezier(94,29)(94,29)(174,29)

\thinlines

\qbezier(70,65)(70,65)(94,29)

\qbezier(71,66)(71,66)(95,30)

\qbezier(70,65)(70,65)(70,56)
\qbezier(70,65)(70,65)(65,58)
\qbezier(70,65)(70,65)(61,63)
\qbezier(70,65)(70,65)(63,70)

\put(38,112){\ensuremath{\scriptstyle P}}

\put(175,31){\ensuremath{\scriptstyle Q}}
\put(50,58){\ensuremath{\scriptstyle R}}

\put(85,26){\ensuremath{\scriptstyle B}}

\put(83,49){\ensuremath{\scriptstyle ?}}

\qbezier(20,100)(45,-5)(180,20)

\qbezier(12,97)(35,-20)(180,20)

\qbezier(30,128)(20,116)(20,100)
\qbezier(30,128)(12,115)(12,97)

\end{picture}
$$
\captionof{figure}{}
\label{newtwo2}
\end{minipage}
\end{figure}

Now when we paste and cut out, the picture is in Figure \ref{newtwo2}, noting specially that
the edge $PR$ is now open. The argument used to conclude that $RB$ should be open, is no longer valid and
we don't know how to mark $RB$ (however that might be determined for
some configurations of the markings of $b$ and the $s_i$).

It corresponds to a BPS state: the eightfold point $P$ lies on a spectral network curve going
in the upward direction, whereas in the cases where the markings of $s_i$ don't tell us what to
do for $b$, the point $R$ is also on a spectral network curve. These two spectral network
curves meet head-on along the segment $PR$ and we get a BPS state.

One can also see that the resulting pre-building can admit folding maps to other
buildings. Here two eightfold points are connected
by an edge (which is not marked ``open''). This is the first and most basic example
where can be a nontrivial choice of maps to buildings,
some of which fold and hence do not preserve distances.

Such a two-manifold construction, which we denote now by $C$, is analogous to the tree with a BPS state
of Figure \ref{tree1}.

$$
\setlength{\unitlength}{.3mm}
\begin{picture}(200,130)

\put(60,60){\circle{2}}
\put(140,60){\circle{2}}

\linethickness{.5mm}
\qbezier(60,60)(100,60)(140,60)
\thinlines

\qbezier(60,60)(60,60)(25,110)

\qbezier(140,60)(140,60)(175,110)

\qbezier(25,110)(10,24)(100,24)
\qbezier(25,110)(10,13)(100,13)

\qbezier(175,110)(190,24)(100,24)
\qbezier(175,110)(190,13)(100,13)

\put(61,64){\ensuremath{\scriptstyle c_1}}
\put(131,64){\ensuremath{\scriptstyle c_2}}

\end{picture}
$$
In this picture there are two sheets below the three edges, and no sheets above.
The two eightfold points $c_1$ and $c_2$ are joined by an edge which we have
emphasized; it corresponds to a BPS state.

This construction $C$ is itself a pre-building, so it can be completed to a building $C\subset B$ containing
it isometrically as a subset. However, it also admits a map to another construction, folding
along the edge $c_1c_2$. To see this, choose two trapezoids sharing the edge $c_1c_2$, and
paste them together.
$$
\setlength{\unitlength}{.3mm}
\begin{picture}(200,130)

\put(60,60){\circle{2}}
\put(140,60){\circle{2}}

\linethickness{.5mm}
\qbezier(60,60)(100,60)(140,60)
\thinlines

\qbezier(60,60)(60,60)(25,110)

\qbezier(140,60)(140,60)(175,110)

\qbezier(25,110)(10,22)(100,22)
\qbezier(25,110)(10,15)(100,15)

\qbezier(175,110)(190,22)(100,22)
\qbezier(175,110)(190,15)(100,15)

\put(61,64){\ensuremath{\scriptstyle c_1}}
\put(131,64){\ensuremath{\scriptstyle c_2}}

\qbezier(60,60)(60,60)(74,40)

\qbezier(140,60)(140,60)(126,40)

\qbezier(74,40)(100,40)(126,40)

\end{picture}
$$

View this as a single sheet over the trapezoid and two sheets below. The
resulting construction $C'$ is again a pre-building so it can be completed to a building
$C'\subset B'$. Composing with the projection $C\rightarrow C'$ therefore gives a map
to a building $h':C\rightarrow B'$. Now $h'$ folds along the segment $c_1c_2$, and it
is not an isometric embedding.

This family of maps $h'$ is the analogue for $SL_3$ of the family of folding maps of trees pictured
in Figure \ref{newtree}.

\subsection{Keeping track of BPS states}

Our main conjecture is that if there are no BPS states, then an indeterminate case doesn't
occur.  In order to set up the possibility to consider BPS states, we will need to
pull along the spectral network (SN) lines in the process of our sequence of $Z_i$.

Let us consider the SN lines on the $Z_i$ as being additional markings, which
should conjecturally be subject to the
following conditions.

\begin{itemize}

\item
Some points will be marked as singularities.

\item
Each singularity will have a single SN line coming out of it.

\item
Each eightfold point which is
at the end of a postcaustic, that is to say, having a single fold line, will have
a single SN coming out of it. It is either opposite to the fold line, or adjacent to it (this is similar to the picture for
the cases of two fold lines at an eightfold point).

\item
Each eightfold point with zero fold lines, that is to say a BNR point, will have two SN lines coming out
of it.

\item
The eightfold points with two fold lines don't have SN lines.

\item
The SN lines should not intersect.

\item
The SN lines, in the middle, go through regular hexagonal points and they are straight.

\end{itemize}

The steps of our process consist of pasting together two parallelograms and cutting them out.
The two edges which are common to the two parallelograms are fold lines, meeting in a fourfold point.

We choose parallelograms which have
singular points somewhere on the boundary.

Each move will either decrease the total length of the postcaustics i.e. the fold lines, or diminish the number of fourfold points. In the case
of integer edgelengths this would prove the finiteness conjecture. In general we need some other argument.

One of the basic moves, when applied to the first fourfold point along a postcaustic, will glue together some new open edges.
In the case that the SN curve was opposite the fold line, the new glued edges are added to the SN curve, this way we
build out the SN curve. In the adjacent case, on the contrary, the SN curve gets eaten up.

One thing that can happen here is that as the SN curve gets eaten up, we can end up passing the singularity. Then it changes to
the other ``opposite'' case and we start spinning out the SN curve in the opposite direction. This is why all the SN directions
at a singular point can actually have the opportunity to contribute.

The main problem is to analyze the transformations of this picture which can occur when the various types of singularities
occur on the boundary of our parallelogram.

The markings $o$ or $f$ of the new edges glued in after we cut out the two parallelograms, are mostly determined; however, because
of the singularity there is one segment where it isn't determined. This was illustrated in the
previous subsection.

At least in a first stage, we will probably also want to assume some genericity of the positions of the singularities so that we don't reach several
singularities on the boundary at once.

The basic idea is to say that if we reach a position where the new markings $o$ or $f$ of the indeterminate edge are not
well determined by the configuration, then it must be that we had two SN curves which join up, and these form a BPS state. This phenomenon was illustrated in the example of Figures \ref{pasting2} and \ref{newtwo2}.

It is possible to envision finishing the process with no more fourfold points, but still having
a segment marked ``fold'' joining two eightfold points. In this case, the
pre-building is not rigid and there might be several harmonic mappings inducing different distance
functions. We are conjecturing that this can only occur if there is a BPS state, indeed
the fold line between eightfold points should correspond to a piece of the BPS state.

Understanding all of the possible cases currently looks complicated, but we hope that it can be done. We formulate
the resulting statement as a conjecture.

\begin{conjecture}
\label{mainconj}
Suppose that $(X,\phi )$ has no BPS states. Then the sequence of two-manifold constructions $Z_i$ is well-defined until
we get to one which has no more fourfold points, or sixfold points with triple fold lines.
Furthermore, this process stops in finite time, at least locally on $X$.
Finally, at the end of the process there are no more ``fold'' markings in the scaffolding.
\end{conjecture}

Associated with this is a finiteness conjecture.

\begin{conjecture}
\label{finitenessconj}
On bounded regions the process stabilizes in finitely many steps.
\end{conjecture}

These conjectures say that the series of two-manifold constructions will stabilize to
a nonpositively curved one which can be called the {\em core}. It plays a role analogous
to Stallings' core graphs \cite{Stallings, KapovichMyasnikov, ParzanchevskiPuder}. Starting
from the core we can then put back everything that had been trimmed off,
following the discussion in Section \ref{sec-recovering}, to get a pre-building.

\begin{corollary}
\label{prebldg}
If there are no BPS states, then on a bounded region we obtain a well-defined pre-building
$B^{\rm pre}_{\phi}$. It has a $\phi$-harmonic map
$$
h_{\phi}: X\rightarrow B^{\rm pre}_{\phi}
$$
such that if $h:X\rightarrow B$ is any $\phi$-harmonic map to a building, then there exists a unique
factorization through a non-folding map $f:B^{\rm pre}_{\phi}\rightarrow B$.
\end{corollary}

\subsection{Distances}

One of the main statements which we would like to obtain through this theory is that the factorization
from the pre-building is an isometric embedding. Recall that the {\em Finsler distance} is defined on
an apartment $A$ with coordinates $x_1,x_2,x_3$ subject to $\sum x_i=0$, by
$$
d((x_1,x_2,x_3), (x'_1,x'_2,x'_3)):= {\rm max} \{ x'_i- x_i\} .
$$
It is not symmetric, and indeed the two numbers $d(x,x')$ and $d(x',x)$ serve to define the {\em vector
distance} which is the series of numbers $x'_i-x_i$ arranged in nonincreasing order.

\begin{conjecture}
\label{isoemb}
In the situation of Corollary \ref{prebldg},
given a $\phi$-harmonic map $h:X\rightarrow B$ to a building, the factorization
$f: B^{\rm pre}_{\phi}\rightarrow B$ is an
isometric embedding for the Finsler distance.
\end{conjecture}

To approach this statement, we should define a notion of {\em noncritical path} in a
construction, with respect to a scaffolding.

Suppose $F$ is a construction provided with scaffolding, and
suppose $S$ is a straight segment joining its two endpoints $x$ and $y$.
Let $s_x$ and $s_y$ denote the vertices corresponding to the directions of $S$
in $F_x$ and $F_y$ respectively. We say that $F$ is {\em Finsler concave} along $S$ if
there exist vertices $u$ and $v$ in $F_x$ and $F_y$ respectively, such that
$u$ is at distance $4$ from $s_x$ and $v$ is at distance $4$ from $s_y$.
Include also the possibility of a ``segment of length zero'' at a point $x$
with two vertices $u,v$ of $F_x$ at distance $\geq 5$. Say that $F$ is {\em Finsler nonconcave}
if there is no segment or point at which it is Finsler concave.

The idea to show Conjecture \ref{isoemb} would be to do the following steps.

\begin{itemize}

\item
The Finsler (or vector) distance between points in a (pre)building
is calculated by adding up the distances along a noncritical path.

\item
If $(F, \sigma )$ is a construction with scaffolding and
$F\rightarrow B$ is a compatible
map to a building, then the image of a noncritical path in $F$ is a noncritical
path in $B$ with the same Finsler length.

\item
In a connected construction with fully open scaffolding, which is Finsler nonconcave, any two points
are joined by a noncritical path.

\item
When there are no BPS states and assuming the previous conjectures, the pre-building
$B^{\rm pre}_{\phi}$ is Finsler nonconcave.

\end{itemize}
It should be possible to check these properties for any given example.

\subsection{Isoperimetric considerations}

We close this section by mentioning an important aspect. Under the pasting-together process,
we glue together two parallelograms which are originally joined along two edges. The two parallelograms
are supposed to be disjoint except for sharing these two edges. We need to choose a maximal
such situation. In particular, when increasing the size of the parallelogram up to its
maximal value (determined by when we meet some kinds of singularities along the boundary)
we would like to make sure that we don't suddenly hit two points which are already previously
identified. This is insured using an isoperimetric argument. Namely, suppose we have pasted
together the parallelograms just up to a point right before the extra joined points.
Cut out this pasted parallelogram, leaving as usual a two-manifold construction. The
eightfold point from the vertex of the pasted parallelogram, is very near to a pair of points
on the two sheets which are already identified. This gives a very short loop in our
two-manifold construction (by very short here, it means arbitrarily short depending on
how close we got to the already-joined points).

By hypothesis, our Riemann surface $X$ is simply connected (if we started with an original
problem on a non-simply connected surface the first step would have been to pass to the
universal cover). It follows that the successive two-manifold constructions $Z_i$ are
also simply connected. So, we are now at a simply connected two-manifold construction
which has a very short  loop. This loop cuts the surface into two pieces, at least
one of which is a disk. We claim that this is impossible. To see that, notice
that since we are going arbitrarily close to the joined-together point,
there are only two possible ways for a post-caustic curve of fold lines to pass
from outside the disk to inside the disk: either through the eightfold vertex of the
pasted parallelogram, or through the joined-together point. From this maximum
of two post-caustics passing through the boundary of the disk, and using the
fact that along any post-caustic the singular points alternate eightfold and fourfold
points starting with eightfold points at the endpoints, we count that the
difference
$$
(\mbox{number of fourfold points}) - (\mbox{number of eightfold points})
$$
is at most $1$. It follows that the total angle deficit due to the curvature inside the
disk is at most $120^{\circ}$ of positive curvature (whereas an arbitrary amount of
negative curvature). Therefore the isoperimetric inequality for this disk cannot be
worse than that of a single fourfold point, where a small boundary implies that
the disk is small itself. We get a contradiction to the arbitrarily small nature of
our loop. This completes the proof that any two parallelograms to be pasted together through
our process, are joined solely on the two edges coming out of the fourfold point.

\section{The A2 example}
\label{sec-a2}

We include a few pictures concerning the next class of examples after BNR.
It illustrates the phenomenon of crossing of caustics, and there are angular rotations leading
to a BPS state.
Because of space considerations we don't have room for a full analysis here,
that will be done elsewhere.

The spectral network for this example is shown in Figure \ref{a2sn}. Note that there are two
caustics intersecting in the middle.

\begin{figure}
\centering
\includegraphics[width=0.9\textwidth]{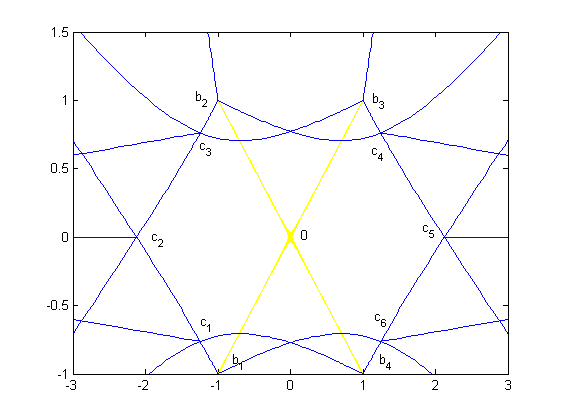}
\caption{A spectral network}
\label{a2sn}
\end{figure}

\begin{figure}
\begin{minipage}[h]{0.85\textwidth}
\centering
\includegraphics[width=0.9\textwidth]{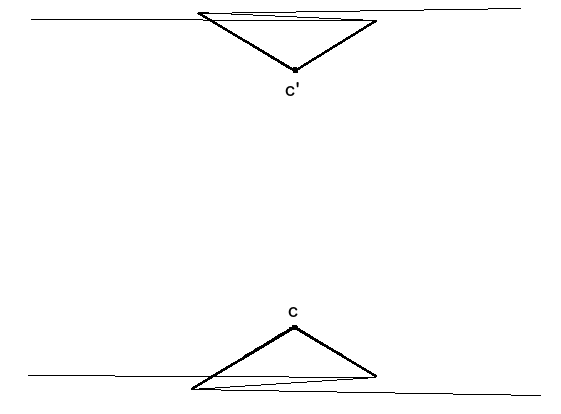}
%\captionof{figure}{}
%\label{a2pb}
\end{minipage}
\caption{The pre-building}
\label{a2pb}
\end{figure}

\medskip
The associated two-manifold construction is, in this case, the same as the pre-building. It
will look as shown in Figure \ref{a2pb}.

There are two eightfold points $c$ and $c'$. The zig-zag boundary lines are shown to
emphasize that the picture represents three sheets over the triangular regions and a single
sheet elsewhere.

The image of the map from $X$ to the pre-building is shown in Figure \ref{phihar}.
At the top is again the spectral network, reprising Figure \ref{a2sn}, with some horizontal dashed lines introduced for visualization.
Below it is the image, in the pre-building that was pictured before.
One can follow the horizontal dashed lines that loop around in order to understand what is happening.

\begin{figure}
\begin{minipage}[h]{0.85\textwidth}
\centering
\includegraphics[width=0.9\textwidth]{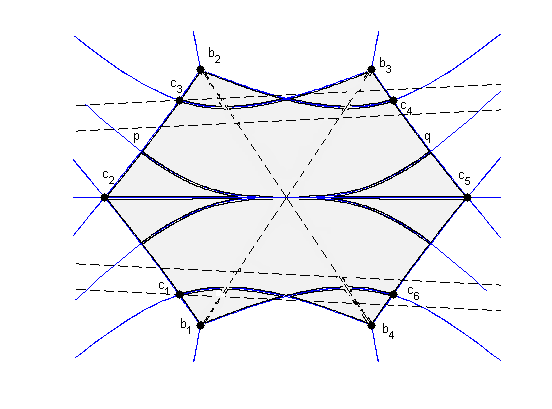}
%\captionof{figure}{}
%\label{a2dom}
\end{minipage}

$$
\downarrow
$$

\begin{minipage}[h]{0.85\textwidth}
\centering
\includegraphics[width=0.9\textwidth]{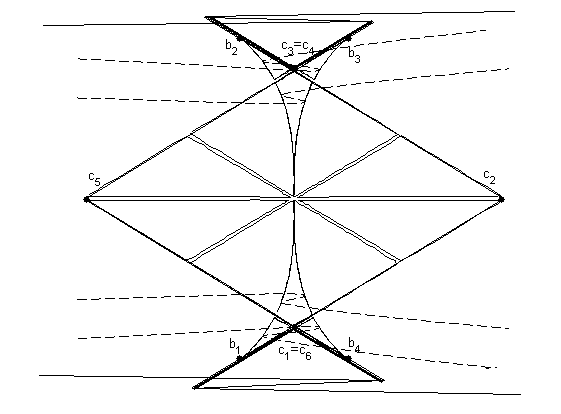}
%\captionof{figure}{}
%\label{a2im}
\end{minipage}

\caption{The $\phi$-harmonic map}
\label{phihar}
\end{figure}

\bigskip

The collision points $c_1$ and $c_6$ go to the point labeled $c$ in the prebuilding; the
collision points $c_3$ and $c_4$ go to the point labeled $c'$. The collision points $c_2$ and
$c_5$ don't go to singular points in the pre-building. Note how the regions containing points
$c_2$ and $c_5$ are folded over to opposite sides as compared to the domain
picture.

\medskip

As the spectral differential $\phi$ changes, for example just by being multiplied by an angular constant $e^{i\theta}$ or some more complicated deformation, the relative positions of the two singular points $c$ and $c'$ will change. The BPS states occur when the line segment $cc'$ goes in the direction of one
of the reflection hyperplanes. We are then in the situation that was considered above in
Subsection \ref{sec-indeterminate}.

\section{Consequences and further considerations}
\label{sec-further}

In this section we look at some consequences and other further considerations.

\subsection{The universal pre-building}

As pointed out in Corollary \ref{prebldg},
the process conjectured above leads to the construction of a map to a
pre-building
$$
h_{\phi} : X\rightarrow B^{\rm pre}_{\phi}.
$$
In the absence of BPS states, it will
have a strong universal property: if $h:X\rightarrow B$ is any $\phi$-harmonic map to a building, then there is a unique
factorization $B^{\rm pre}_{\phi}\rightarrow B$ which is an isometric embedding.

Conjectures \ref{mainconj} and \ref{finitenessconj}, when there are no BPS states,
imply that we can construct the pre-building without necessarily following
through the whole procedure. Indeed, if we can illustrate some pre-building $B'$
with a $\phi$-harmonic
map $X\rightarrow B'$, such that some choice of the initial construction surjects onto $B'$,
then $B'=B^{\rm pre}_{\phi}$. For example, this shows that the picture
given Figure \ref{a2pb} of Section \ref{sec-a2} does indeed show the pre-building for the A2 example.

\subsection{Determination of the WKB exponents}

Our conjectures up through Conjecture \ref{isoemb}
imply that the WKB exponents are uniquely determined, as we state in the following
corollary.

\begin{corollary}
For any WKB problem with spectral curve corresponding to $\phi$, the ultrafilter exponent $\nu ^{WKB}_{\omega}(P,Q)$
for the transport from a point $P$ to a point $Q$, must be equal to the Finsler distance in
$B^{\rm pre}_{\phi}$
from $h_{\phi}(P)$ to $h_{\phi}(Q)$. In particular, it doesn't depend on the
choice of ultrafilter, which implies that the limit used to define the WKB exponent exists and its value
is calculated as the Finsler distance in  $B^{\rm pre}_{\phi}$.
\end{corollary}

Even before obtaining general proofs of all of the conjectures which go into this statement, it should
be possible to verify the required things when we are given a particular example of a spectral
differential $\phi$, thus obtaining a calculation of the WKB exponents.

One expects, from the resurgence picture, that when there is a BPS state the ultrafilter exponent should usually depend on
the choice of ultrafilter.

\begin{conjecture}
The WKB exponents determined by our process, when there are no BPS states, coincide with the
exponents constructed by the spectral network procedure of Gaiotto-Moore-Neitzke.
\end{conjecture}

\subsection{Relation with the work of Aoki, Kawai, Takei {\em et al}}

The pictures we consider here are starting to look a lot like the pictures
which occur in the work of Aoki, Kawai, and Takei \cite{AKT1} \cite{AKT2}. It will be interesting
to establish a precise comparison.

\subsection{Stability conditions}

In the works of Bridgeland and Smith \cite{BridgelandSmith} and Haiden, Katzarkov and Kontsevich \cite{HaidenKatzarkovKontsevich}, a picture is developed whereby each pair $(X,\phi )$
consisting of a compact Riemann surface and a twofold spectral covering,
corresponds to a Bridgeland stability condition on a Fukaya-type triangulated category.
The space of $(X,\phi )$ is divided up into chambers whose boundaries are walls determined by
the existence of BPS states. A stability condition corresponding to $(X,\phi )$ should really be thought of as associated to the circle of $(X,e^{i\theta}\phi )$, and the values of $\theta$ where BPS states arise are
the angles of the central charge for stable objects.

This chamber and wall structure is expected to be closely related to the wallcrossing phenomenon introduced
by Kontsevich and Soibelman \cite{KontsevichSoibelman}. Making this relationship precise, and extending the
construction of stability conditions to higher rank spectral curves, are some of the long-term motivations for
the present project.

In the $SL_2$ case these subjects have been treated recently by Iwaki and Nakanishi \cite{IwakiNakanishi},
and one of the goals in the higher rank case would be to extend the various different aspects of their theory.

What we obtain from our construction for $SL_3$
is that inside each chamber, that is to say inside a connected region
where there are no BPS states, our pre-building $B^{\rm pre}_{\phi}$
is a combinatorial geometric object
associated\footnote{The pre-building is essentially uniquely defined according to our conjectures.
The only ambiguity comes from the choice of size of neighborhoods used to make the initial construction.
Those affect the size, which is inessential, of the additional
``stegosaurus dorsal plates'' attached to a caustic,
illustrated for example in the white areas just above the dotted line in Figure \ref{bnrPB}.
}
to $(X,\phi )$. Over a given chamber the pre-building will vary in a smooth way. As we cross a wall determined by a BPS state, the pre-building will transform in some kind of ``mutation''. The transformations of spectral
networks are pictured in \cite{GMN-SN}.

We hope that the study of the geometric form and structure of these mutations
can provide some insight into the
structure of the categorical stability conditions.

Kontsevich has outlined a far reaching program that associates a ``Fukaya-type category with coefficients'' to a ``perverse sheaf of categories'' on a symplectic manifold $X$.  A major part of this program is the construction of stability structures on the Fukaya category with coefficients in a sheaf of categories starting from the data of stability structures on the stalks of this sheaf.  A symplectic fibration $\pi: Y \rightarrow X$  conjecturally gives rise to such a perverse sheaf of categories $\mathfrak{S}$, whose stalks are, roughly speaking,  the Fukaya categories of the fibers of $\pi$. In this situation, the Fukaya category with coefficients in the sheaf of categories $\mathfrak{S}$ on $X$ is supposed to be equivalent to a Fukaya type category associated to $Y$. Recently, Kapranov and Schechtman \cite{Kapranov-Schechtman} have given a precise definition of the notion of a ``perverse sheaf of categories'', which they call a {\em perverse Schober}, in the case when $X$ is a Riemann surface.

There is a CY-manifold $Y$ of complex dimension 3 associated to a spectral cover $\Sigma$ of a Riemann surface $X$, namely, the conic bundle over the cotangent bundle $T^{\vee}_{X}$  whose discriminant locus is $\Sigma$. The natural morphism $\pi: Y \rightarrow X$ is a symplectic fibration, and therefore should define a perverse Schober $\mathfrak{S}$ on $X$. We formulate the following:

\begin{conjecture}
The Schober $\mathfrak{S}$ on $X$ associated with the conic bundle, together with a stability structure on it, can be recovered from a family $\{h_{\theta}\}_{\theta \in \mathbb{R}}$ of versal harmonic $e^{i \theta} \phi$-maps, where $\phi$ is the multivalued differential corresponding to the given spectral cover $\Sigma$. 
\end{conjecture}

\subsection{Higher ranks}

In this paper we have restricted to looking at buildings for $SL_3$. These two-dimensional
objects remain geometrically close to the points of the original Riemann surface, and this
has been used in fundamental ways in our discussion. So it remains a big open problem how to
go to a general theory for buildings of arbitrary rank, say for the groups $SL_r$. The image of
$X$ becomes a very thin subset. Maxim Kontsevich suggested on several occasions that we should think of trying to ``thicken'' $X$ to something of the right dimension.

We feel that this will indeed be the way that the
theory should play out. The initial construction will give a neighborhood of $X$ of the correct dimension
$r-1$.
The $r-2$-dimensional
boundary of this neighborhood looks something like a bundle of $r-4$-spheres over the $2$-dimensional
$X$. Probably, the analogue of our process described above will be some kind of discrete evolution
of the boundary as we successively enlarge the region of the building that is being constructed.
The two-manifold constructions which show up in our process for $SL_3$ should be thought of as the boundaries
of what are, for us, ``$0$-dimen\-sional disk bundles over $X$''.

\subsection{Buildings as enriched categories}
In \cite{Lawvere}, Lawvere made the fundamental observation that metric spaces can be viewed as categories enriched over the poset $\left( [0,\infty], \geq \right)$ equipped with the (symmetric) monoidal structure given by addition of real numbers. Furthermore, various metric constructions, such as Cauchy completion, have natural categorical interpretations. Motivated by these observations, Flagg \cite{Flagg} has shown that much of the basic theory of metric spaces carries over to categories enriched over an arbitrary value distributive quantale $\mathcal{Q}$. The distance between two points in these ``generalized metric spaces'' is an object of the quantale $\mathcal{Q}$ instead of a real number.

Tits' definition of a (discrete) building as a chamber system with a Weyl group valued ``distance function'' (see, e.g.\ \cite{AbramenkoBrown}) strongly suggests that the theory of buildings should naturally be situated within the theory of $\mathcal{Q}$-enriched categories for some suitable quantale $\mathcal{Q}$ that depends on the Weyl group. This idea has been put forth by Dolan and Trimble, and is discussed in \cite{Trimble}. 

In a forthcoming work, we hope to systematically develop a theory of $\mathbb{R}$-buildings as categories enriched over a suitable variant of the category of enclosures considered in this paper. We feel that, combined with the theory of buildings as sheaves on the site of enclosures discussed earlier, such a theory will provide a powerful framework in which to formulate and study the types of constructions we have looked at in this paper in connection with the problem of constructing versal buildings.

\end{document}